\documentclass[a4paper,11pt,leqno]{amsart}
\usepackage{amssymb, amsmath, amsthm, graphicx, color, epsfig, amsfonts, setspace,dsfont, hyperref}

\newcommand{\mathsym}[1]{{}}
\newcommand{\unicode}[1]{{}}

\newtheorem{thm}{Theorem}[section]
\newtheorem{lem}[thm]{Lemma}
\newtheorem{cor}[thm]{Corollary}
\newtheorem{prop}[thm]{Proposition}
\newtheorem{defn}[thm]{Definition}
\newtheorem{rem}[thm]{Remark}

\newtheorem{exa}[thm]{Example}
\numberwithin{equation}{section}

\renewcommand{\a}{\alpha}
\renewcommand{\b}{\beta}

\newcommand{\de}{\delta}

\newcommand{\ga}{\gamma}

\newcommand{\si}{\sigma}

\newcommand{\Ga}{\Gamma}

\def\R{{\mathbb{R}}}
\def\RR{{\mathcal{R}}}

\def\Z{{\mathbb{Z}}}
\def\T{{\mathbb{T}}}
\def\D{\mathcal{D}}
\def\G{\mathcal{G}}
\def\E{\mathcal{E}}
\def\q{q_3}
\def\r{q_1}

\allowdisplaybreaks

\title{Coupled KPZ equations and their decoupleability}

\author{Boliang Fu}
\address{Beijing Institute of Mathematical Sciences and Applications, \\
No.\ 544, Hefangkou Village, Huaibei Town, Huairou District, \\
Beijing 101408, China}
\email{{\tt 495414401@qq.com}}

\author{Tadahisa Funaki}
\address{Beijing Institute of Mathematical Sciences and Applications, \\
No.\ 544, Hefangkou Village, Huaibei Town, Huairou District, \\
Beijing 101408, China}
\email{{\tt funaki@ms.u-tokyo.ac.jp}}

 \author{Sunder Sethuraman}
\address{Department of Mathematics\\
University of Arizona\\
621 N. Santa Rita Ave.\\
Tucson, AZ 85750, USA}
\email{{\tt sethuram@arizona.edu}}

 \author{Shankar Venkataramani}
\address{Department of Mathematics\\
University of Arizona\\
621 N. Santa Rita Ave.\\
Tucson, AZ 85750, USA}
\email{{\tt shankar@arizona.edu}}

\date{\today}

\begin{document}


\begin{abstract}
We discuss characterizations of the decoupleability, partial and full, of trilinear or completely symmetric real $n\times n\times n$ tensors, which inform on the structure of certain coupled KPZ equations.  Informally, when the tensor is partially decoupleable, one of the components in the coupled KPZ equation splits off from the others, while when the tensor is fully decoupleable, each of the $n$ components splits off from the others.  

Such a characterization is recast as a problem of membership of trilinear tensors in $O(n)$ orbits of subsets of fully decoupleable and partially decoupleable tensors.  When $n=2$, we show these subsets are the same, and in this case give a single criterion in terms of the entries of a tensor for membership in the orbits of these subsets.  When $n\geq 3$, the subsets are different.  For $n\geq 3$, we characterize full decoupleability in terms of several abstract relations, which when $n=3$ are made explicit.  When $n=3$, we also explicitly characterize partial decoupleability.

The methods involve notions in applied invariant theory, relating $O(n)$ invariant subsets to stabilizer subgroup actions on smaller sets.  When $n=3$ make use of the explicit basis of invariants found by Olive and Auffray.  When $n=2$, we also supply two other more direct arguments.
\end{abstract}

\maketitle

\tableofcontents

\section{Introduction}

We consider the coupled KPZ equation for $h=(h^i(t,x))_{i=1}^n$ with
$n$-components on $\T=[0,1)$ 
with periodic boundary condition or on $\R$ written in a canonical form
\begin{equation} \label{eq:1.1}
\partial_t h^i = \tfrac12 \Delta h^i + \tfrac12 \Ga_{jk}^i
\nabla h^j \nabla h^k + \xi^i,\quad 1\le i \le n,
\end{equation}
where $\Delta = \partial^2_x$, $\nabla = \partial_x$, and
$\xi=(\xi^i(t,x))_{i=1}^n$ is an $\R^n$-valued
space-time Gaussian white noise with covariance structure
$$
E[\xi^i(t,x)\xi^j(s,y)] = \de^{ij}\de(t-s)\de(x-y),
$$
for $t, s\ge 0$ and $x, y \in \T$ or  $\R$.
We use Einstein's summation convention for the second term in the right-hand side of
\eqref{eq:1.1}.  For the coupling constant $\Ga = (\Ga_{jk}^i)$, an $n\times n\times n$ real tensor, without loss of
generality, we assume the bilinear condition
\begin{equation*}
\Ga_{jk}^i= \Ga_{kj}^i \; \text{ for all } \, i, j, k,
\end{equation*}
due to the symmetry of the term $\partial h^j \partial h^k$ in $j, k$. 

 We also
introduce a stronger condition for $\Ga$ called the trilinear condition
\begin{equation}\label{eq:T}
\Ga_{jk}^i= \Ga_{kj}^i = \Ga_{ik}^j\;\text{ for all } \, i, j, k. 
\end{equation}
A trilinear $\Gamma$ is sometimes also called `completely symmetric'.
We will define 
$$\mathcal{T}_n := \{\Ga = (\Ga_{jk}^i)_{i,j,k=1}^n: {\rm \ trilinear }\}.$$
We mention references for the general theory of tensors in various contexts include \cite{Boehler}, \cite{Cox-Little-Oshea}, \cite{Derkson}, \cite{Gnang}, \cite{Kolda}, \cite{Popov}, \cite{Qi}, \cite{Sturmfels},  \cite{Weyl}.

The coupled KPZ equation is ill-posed in a classical sense and requires 
a renormalization.  The local-in-time well-posedness is shown
under the bilinearity of $\Ga$ by applying regularity structures or paracontrolled calculus
for singular SPDEs due to Hairer \cite{Hairer} or Gubinelli-Imkeller-Perkowski \cite{GIP}, respectively.
The trilinear condition plays an important role.  Indeed, assuming \eqref{eq:T},
one can prove several results on $\T$ including
(1) the global-in-time well-posedness, more precisely, the existence and
uniqueness of solutions for all initial values in H\"older-Besov space 
$\mathcal{C}^\a=(\mathcal{B}_{\infty,\infty}^\a(\T))^n, \a<1/2$,
(2) strong Feller property (Hairer-Mattingly \cite{HM16}), (3) the unique invariant 
measure (except for a shift in $h$) is given by the periodic Wiener measure on $C(\T,\R^n)$,
(4) lack of necessity (or cancellation) of a logarithmic renormalization 
for fourth order terms and (5) the clarification of the difference of limits of two types of 
approximations originally introduced by \cite{FQ} when $n=1$.  See \cite{FH17} 
for details of these results. 

Motivated by the study of `nonlinear fluctuating hydrodynamics' in one dimensional systems with several conservation laws \cite{Schutz}, \cite{Spohn1}, the article \cite{BFS} derives a coupled KPZ-Burgers equation, formally associated to the gradient $\chi=\nabla h$ with respect to \eqref{eq:1.1}, namely
 \begin{align*}
 \partial_t \chi^i = \tfrac12 \Delta \chi^i + \tfrac12 \Ga_{jk}^i
\nabla\big(\chi^j \chi^k\big) +\nabla \xi^i,\quad 1\le i \le n,
 \end{align*} 
 from
multi-species zero-range processes for the system of fluctuation fields associated to each species.  The trilinear condition is proven for 
the coupling constant $\Gamma$ of the coupled equations obtained in this way and written in a canonical form.  This coupled system in \cite{BFS} is derived under the assumption of equal `characteristic velocities' for each of the types.  We mention, if these velocities are different, then \cite{CGMO} for $n=2$ derive a system of independent KPZ-Burgers equations for the fluctuation fields seen in each characteristic frame; see also \cite{FeGe}.  It has been shown in \cite{RDKKS} that the general coupled $n=2$ system belongs to the KPZ class in terms of the $1:2:3$ scaling.  

The purpose of this article is to understand when $\Gamma$ is partially or fully decoupleable in $n\geq 2$.  Such information would inform on the structure of the KPZ coupled system discussed earlier.
We will give a general characterization of fully decoupled tensors in $n\geq 2$, as an application of a more abstract scheme.  Then, we will concentrate on $n=2,3$ where more explicit characterizations of both fully decoupleable and partially, but not fully, decoupleable tensors are made.

We formulate the problems more carefully in Section \ref{formulation section} (see Section \ref{two and three components section} for explicit formulations when $n=2, 3$) and then in the rest of the Introduction discuss our results and methods, primarily involving applied invariant theory, in Section \ref{discussion section}.

\subsection{Formulation of the problem}
\label{formulation section}

If $\si=(\si_{ij})$ is an orthogonal $n\times n$ matrix, $\si\xi$ remains $\R^n$-valued
space-time Gaussian white noise in law. Thus, under the transformation 
$\tilde h_t :=\si h_t$, the vector $\tilde h_t$ satisfies \eqref{eq:1.1} in law,
with $\Ga$ changed to $\si\circ\Ga$ defined by
\begin{equation*}
(\si\circ\Ga)_{jk}^i := \sum_{i',j',k'}
\si_{i i'} \Ga_{j'k'}^{i'} \si_{j'j}^{-1}\si_{k'k}^{-1}.
\end{equation*}
Since $\si$ is an orthogonal matrix, i.e.\ $\si^{-1}_{j'j} = \si_{jj'}$,
we have
\begin{equation} \label{eq:1.3}
(\si\circ\Ga)_{jk}^i = \sum_{i',j',k'}
\si_{i i'} \Ga_{j'k'}^{i'} \si_{jj'}\si_{kk'}.
\end{equation}

This shows that $\si\circ\Ga$ is trilinear if $\Ga$ is trilinear and $\si$ is an
orthogonal matrix, that is $\si\in O(n)$, the orthogonal group.  In other words, the trilinearity is kept under the rotation 
and reflection.  Note that the bilinearity is kept for any regular matrix $\si$
under the transform $\si\circ\Ga$.

Recall that $SO(n)$ is the subgroup of `rotations' in $O(n)$ with determinant $1$.  In fact, $O(n)$ is the semidirect sum of $SO(n)$ and any subgroup formed with the identity and a `reflection', that is an element of $O(n)$ with determinant $-1$.

We now consider a useful map which gives an equivalent formulation: For $x = (x_1,\ldots,x_n) \in \R^n$, define
$$
f(x;\Ga) := \sum_{i,j,k=1}^n \Ga_{jk}^i x_ix_jx_k.
$$
Such a map is in $1:1$ correspondence with trilinear tensors $\{\Gamma^i_{j,k}\}_{i,j,k=1,\ldots, n}$:  Given such a homogeneous cubic polynomial $f$, we may compute $\Gamma^i_{j,k}$ as
$\Gamma^i_{j,k} = \frac{1}{3!}\partial_i \partial_j \partial_k f$.

Equivalently, we can also find the trilinear tensor $\Gamma$ via the relation
$$\Gamma(X,Y,Z)=\sum_{i,j,k=1,\ldots, n}X_iY_jZ_k\Gamma^i_{j,k} = \frac{1}{3!}X\cdot \nabla(Y\cdot\nabla(Z\cdot\nabla f)),$$
for $X,Y,Z\in \R^n$.  Here, one may interpret $\Gamma(X,Y,Z)$ as a `lifting' of $f(x;\Ga)$ as  $\Gamma(X,Y,Z)= f(x;\Ga)$ when $X=Y=Z=x$.  Note that the definition of $\Gamma(X,Y,Z)$, as it is a scalar with respect to $x$, doesn't depend on the argument $x\in \R^n$ of $f(x)$.

One may relate the action of $\sigma\in O(n)$ on $\Gamma$ to that of $\sigma$ acting on $x\in \R^n$
with respect to $f$.

\begin{lem}  \label{lem:4.1}
For $\Gamma\in \mathcal{T}_n$ and $\si\in O(n)$, we have $f(x;\si\circ\Ga) = f(\si^{-1}x;\Ga)$.
\end{lem}

\begin{proof}
Noting that $\si_{ii'} = (\si^{-1})_{i'i}$, we have
\begin{align*}
f(x; \si\circ\Ga)
& = \sum_{ijk}\sum_{i'j'k'} \si_{ii'} \Ga_{j'k'}^{i'} \si_{jj'}\si_{kk'} x_ix_jx_k \\
& = \sum_{i'j'k'} \Ga_{j'k'}^{i'} (\si^{-1}x)_{i'} (\si^{-1}x)_{j'}(\si^{-1}x)_{k'} \
 = \ f(\si^{-1}x; \Ga). \qedhere
\end{align*}
\end{proof}
We comment in passing in other problems, when such a homogeneous function $f$ is introduced first, the specification of $f(\sigma^{-1}x; \Gamma)$ would then define the action $\sigma\circ \Gamma$.

We state and recall the definition of full and partial decoupleability for the coupled
KPZ equation \eqref{eq:1.1} with trilinear $\Gamma$.

\begin{defn}[cf.\ Definition 8.1 of \cite{BFS}]

{\rm (1)}
We say that the KPZ-system \eqref{eq:1.1} is fully decoupleable 
if there exists $\sigma\in O(n)$ such that for any 
$i \in \{1,\ldots, n\}$, the coupling constants $(\si\circ\Ga)_{jk}^i$ 
are zero for any $(j,k) \ne (i,i)$. \\
{\rm (2)}
We say that it is partially decoupleable if there exists
$\sigma\in O(n)$ such that there exists $i \in \{1,\ldots, n\}$ for which the coupling
constants $(\si\circ\Ga)_{jk}^i$ are zero for any $(j,k) \ne (i,i)$. 
\end{defn}

We now define the orbits in $\mathcal{T}_n$ with respect to $O(n)$. 

\begin{defn}
\label{FD PD defn}
For $\Ga_1, \Ga_2 \in \mathcal{T}_n$, we say $\Ga_1 \sim \Ga_2$ if
there exists $\si \in O(n)$ 
such that $\Ga_2 = \si\circ \Ga_1$.  Define $\{\sigma\circ \Gamma: \sigma \in O(n)\}$ as the $O(n)$-orbit of $\Gamma\in \mathcal{T}_n$.  Then, $\Ga_1$ and $\Ga_2$ are on the same $O(n)$ orbit exactly when $\Ga_1\sim \Ga_2$.
\end{defn}

It is evident that 
`$\sim$' is an equivalence relation, and one can consider the quotient space $\mathcal{T}_n/\sim$ of `orbits', or write $\mathcal{T}_n$ as a union of orbits.

In the language of the relation $\sim$, we may restate Definition \ref{FD PD defn} as the membership of $\Gamma$ in sets of fully and partially decoupled tensors, $\mathcal{T}_{n, \rm FD}$ and $\mathcal{T}_{n, \rm PD}$:
For $1 \le i \le n$ and $\beta\in \R$, let $\G^{(i,\beta)}$
be the $n\times n$ matrix such that $(\G^{(i,\beta)})_{ii} = \beta$ and $(\G^{(i,\beta)})_{jk}=0$
otherwise.  For $\beta_1,\ldots,\beta_n \in \R$, let $\Ga \equiv \Ga^{(\beta_1,\ldots,\beta_n)}
:= (\Ga^i = \G^{(i,\beta_i)})_{i=1, \ldots, n}$, where we denote $\Gamma^i = (\Gamma^i_{jk})_{1\leq j,k\leq n}$.  By definition,
\begin{align*}
& \mathcal{T}_{n, \rm FD} := \{\Ga\in \mathcal{T}_n; 
\Ga\sim \Ga^{(\beta_1,\ldots,\beta_n)} \; \text{ for some } \beta_1,\ldots,\beta_n \in \R\},\\
& \mathcal{T}_{n, \rm PD} := \{\Ga\in \mathcal{T}_n; 
\Ga\sim (\Ga^n= \G^{(n,\beta)}, \Ga^1,\ldots, \Ga^{n-1})\\
&\ \ \ \ \ \ \ \ \ \ \ \ \ \ \ \ 
\text{ for some } \beta\in \R \text{ and } n\times n \ {\rm real \ matrices \ }\Ga^1,\ldots,\Ga^{n-1} \}.
\end{align*}
We comment, in the definition of $\mathcal{T}_{n, \rm PD}$, that we have taken $i=n$ by an a priori $O(n)$ transformation.
Note also that some components of $\Ga^1,\ldots, \Ga^{n-1}$ are determined from $\Ga^n$.

Alternatively, in the language of third order homogeneous polynomials $f$, tensors $\Gamma\sim \Ga'$ exactly when there is a $\sigma\in O(n)$ such that $f(x;\Ga') = f(x; \sigma\circ\Ga) (=f(\sigma^{-1}x; \Ga))$. Moreover, one may give an equivalent definition of the class of fully decoupled and partially decoupled tensors $\Gamma$:
Namely, $\Gamma \in \mathcal{T}_{n, \rm FD}$ exactly when there is a $\sigma\in O(n)$ such that
\begin{align*}
f(\sigma^{-1}x;\Gamma) = \sum_{\ell=1}^n \Gamma^{\ell}_{\ell, \ell}\big((\sigma^{-1}x)_\ell\big)^3.\end{align*}
On the other hand, $\Gamma \in \mathcal{T}_{n, \rm PD}$ exactly when there is a $\sigma\in O(n)$ such that
\begin{align*}f(\sigma^{-1}x;\Gamma) = \Gamma^n_{n,n}\big((\sigma^{-1}x)_n\big)^3 + \sum_{i,j,k=1,\ldots, n-1}\Gamma^{i}_{j,k}(\sigma^{-1}x)_i (\sigma^{-1}x)_j(\sigma^{-1}x)_k.
\end{align*}

\subsection{Discussion of results and methods}
\label{discussion section}

We may view the problem of characterizing $\mathcal{T}_{n, \rm FD}$ or $\mathcal{T}_{n, \rm PD}$ in a larger context.  We say that a subset $\mathcal{S}\subset \mathcal{T}_n$ is `$O(n)$ invariant' if $\sigma\circ \Gamma \in \mathcal{S}$ whenever $\sigma\in O(n)$ and $\Gamma\in \mathcal{S}$.  By definition, $\mathcal{T}_{n, \rm FD}$ and $\mathcal{T}_{n, \rm PD}$ are both invariant subsets, being equal to $\cup_{R\in \mathcal{R}}\big\{O(n) \ {\rm orbit \ of \ }R\big\}$ for subsets $\mathcal{R}$ consisting of say tensors in reduced fully decoupleable or partially decoupleable forms.
In this sense, our problem is a type of membership problem to determine which tensors $\Gamma\in \mathcal{T}_n$ belong to an `invariant' subset $\mathcal{S} = \cup_{R\in \mathcal{R}}\big\{O(n) \ {\rm orbit \ of \ }R\big\}$, where $\mathcal{R}$ is the collection of fully or partially decoupled tensors in $\mathcal{T}_n$.

Membership problems have been considered in other contexts. For instance, \cite{Cox-Little-Oshea} discusses the implicitization problem to determine when a point belongs to a set given parametrically, which has many applications.    In such works, Groebner basis computations are often used to deduce relations, usually with respect to the underlying field $\mathbb{C}$ (see Chapter 3 in \cite{Cox-Little-Oshea}).

We can also view the problem of characterizing $\mathcal{T}_{n, \rm FD}$ or $\mathcal{T}_{n, \rm PD}$ as an instance of a decomposition problem, namely we are determining when we can express a homogeneous cubic polynomial as a sum of `simpler' functions of `linear forms' using orthogonal transformations. Related decomposition problems for homogeneous, cubic polynomials $f(x;\Ga)$ over $\R$ or $\mathbb{C}$ with respect to general or special linear groups $GL(n)$ or $SL(n)$,
have been considered in the literature. See for instance \cite{Elliott} p. 263, \cite{Huybrechts} with respect to $GL(n)$ over $\mathbb{C}$ when $n=2$ and $n\geq 3$.  More generally, when $f$ is a $d$-order homogeneous polynomial, the study of representatons $f = \sum_{i=1}^r \lambda_i q_i^d$ where $q_i$ are linear expressions in $x_1, \ldots, x_n$ over $\mathbb{C}$ or $\mathbb{R}$ are of current interest and have been considered in \cite{Sturmfels_openproblems} and references therein.  

In our formulation, we look for membership of real $n\times n \times n$ tensors, as opposed to over $\mathbb{C}$, in $O(n)$-invariant sets of fully or partially decoupled tensors.  It appears that such $O(n)$-invariant sets have not been considered much in the literature. In this sense, our work may be among the first to detail some of their structures.

When $n=2$, we show that the notion of being partially decoupled is the same as being fully decoupled, $\mathcal{T}_{2, \rm PD} = \mathcal{T}_{2, \rm FD}$ (Lemma \ref{lem:2.1}).  We will characterize membership in $\mathcal{T}_{2, \rm FD}$ in several ways:  By directly solving for a required $\sigma \in O(2)$ (Proposition \ref{prop:2.2}), by a more geometric argument involving eigenstructures (Proposition \ref{prop:ode}), and by using `applied invariant theory' (Section \ref{full decoup sym proof} and Remark \ref{rem:6.2}).  The `direct solution' and `geometric' arguments yield transformations $\sigma\in O(2)$ such that $\sigma\circ\Gamma$ is in a fully decoupled canonical form.

When $n\geq 3$, $\mathcal{T}_{n, \rm FD} \subsetneq \mathcal{T}_{n, \rm PD}$ (Proposition \ref{prop:PD neq FD}).  However, even for $n=3$, it is substantially more difficult to directly solve for $\sigma \in O(3)$ such that $\sigma \circ \Ga$ is in fully or partially decoupled form, given the number of parameters and the nonlinear structure of the resulting equations.
We show that, generalizing the applied invariant theory approach (cf. books Cox-Little-O'Shea \cite{Cox-Little-Oshea}, Derksen-Kemper \cite{Derkson}, and Sturmfels \cite{Sturmfels}, among others) used in the case $n = 2$, allows us to treat the problem for $n\geq 3$.  A main ingredient in the solutions when $n=3$ is an `integrity  basis', that is a generating set for the subalgebra of polynomial invariants,
constructed by Olive and Auffray \cite{OA}.

 As discussed in Section \ref{invariants section}, an `invariant' is a real function of the components of a tensor $\Gamma\in \mathcal{T}_n$ that is invariant under $O(n)$ action.  There exists a finite basis $\mathcal{I}$ of polynomial invariants generating all other polynomial invariants (Hilbert's finiteness Theorem 2.1.3 \cite{Sturmfels}), although it is difficult to construct such a  basis in general.  Also, it is known that, for a tensor $\Ga$, the collection of the values of the basis polynomials identifies uniquely its $O(n)$ orbit (Proposition \ref{prop:separating}).  So, to solve a membership problem, we would need to identify the (typically polynomial) relations between the values of the basis invariants that hold for all tensors in $\mathcal{S}$ but fail for any tensor not in $\mathcal{S}$.

In particular, we formulate an abstract characterization (Theorem \ref{thm:meta}) to determine membership of $\Gamma\in \mathcal{T}_n$ in an invariant set $\mathcal{S} = \cup_{R\in \mathcal{R}}\big\{O(n) \ {\rm orbit \ of \ }R\big\}$ for a given set of tensors $\mathcal{R}$, via which the specific results for $\mathcal{S}=\mathcal{T}_{n, \rm FD}$ and $\mathcal{S} = \mathcal{T}_{3, \rm PD}\setminus\mathcal{T}_{3, \rm FD}$ are found.  This characterization formulates membership in $\mathcal{S}$ in terms of a `semi-algebraic' set (cf. Remark \ref{rem:semi-algebraic}), and also identifies a tensor $R\in \mathcal{R}$ in the $O(n)$ orbit of $\Gamma$.  In a `generic' setting, we will also identify maps $\sigma \in O(n)$ such that $\sigma \circ \Gamma = R$ (Theorem \ref{rem:rotation}).

In a sense, the idea is to make use of characterizations on a `small' set $\mathcal{R}$, which say might consist of tensors in a particular form.  
Central to our analysis is the stabilizer subgroup $G_\RR$
consisting of all $g \in O(n)$ such that $g \circ R \in \mathcal{R}$ for all $R \in \mathcal{R}$.
The subgroup $G_\RR$ has its own polynomial invariants, say generated by a collection $\mathcal{J}$.
Since
every $O(n)$ polynomial invariant in $\mathcal{I}$, when restricted to $\mathcal{R}$, is a fortiori $G_\RR$ invariant, we may write each member of $\mathcal{I}$ as a polynomial 
in the generating set $\mathcal{J}$.  
One then finds `lifts' of the $G_\RR$ invariants to suitable $O(n)$ invariants $\widetilde{\mathcal{J}}$ on an invariant subset $\widetilde{\mathcal{T}}\subset \mathcal{T}_n$ which includes $\mathcal{S}$.  Interestingly, such extensions $\widetilde{\mathcal{J}}$ may not be polynomial, in which case $\widetilde{\mathcal{T}}$ may be a proper subset of $\mathcal{T}_n$. 
 Then, we lift the relations between $\mathcal{I}$ and $\mathcal{J}$ on $\mathcal{R}$ to those on $\widetilde{\mathcal{T}}$, by substituting the extensions $\widetilde{\mathcal{J}}$ for $\mathcal{J}$.  These lifted relations, as the values of $\mathcal{I}$ separate $O(n)$ orbits, will be necessary conditions for membership of a tensor $\Gamma$ in $\mathcal{S}$.  

Sufficiency will be provided as long as one can solve for a tensor $\hat R\in \mathcal{R}$ such that the values of $\widetilde{\mathcal{J}}$ on a given $\Gamma$ match those on $\hat R$.   With such `solvability', the values of $\mathcal{I}$ on $\Gamma$ are seen by the necessary relations to equal those on $\hat R$.  We would then conclude, as $\mathcal{I}$ separates $O(n)$ orbits, that $\Gamma$ and $\hat R$ are on the same orbit, that is in $\mathcal{S}$.

We will be able to give implicit characterizations of $\mathcal{T}_{n, \rm FD}$ for all $n\geq 3$ (Theorem \ref{thm:full n}).  However, more explicit determinations are made when $n=3$ for $\mathcal{T}_{3, \rm FD}$ and $\mathcal{T}_{3, \rm PD}\setminus\mathcal{T}_{3, \rm FD}$
 (Theorems \ref{thm:full explicit} and \ref{thm:partial n=3}).
 All of these characterizations of membership involve polynomial relations, as suggested earlier, in terms of a generating set $\mathcal{I}$ of $O(n)$ polynomial invariants of $\mathcal{T}_n$, as well as a `solvability' condition to ensure the relations do not involve extraneous $O(n)$ orbits.  Importantly, as suggested, we make use of the `integrity basis' of polynomial $O(3)$ invariants on $\mathcal{T}_3$ \cite{OA} for the results when $n=3$.

From another point of view, given an $O(n)$ invariant set $\mathcal{S}\subset \mathcal{T}_n$, there may be some flexibility in applying the abstract characterization Theorem \ref{thm:meta}.  There may be choice of sets $\mathcal{R}$ and associated $O(n)$-stabilizer subgroups $G_\RR$ where $\mathcal{S}=\cup_{R\in\mathcal{R}}\big\{G_\RR \ {\rm orbit \ of \ }R\big\}$.  If $\mathcal{R}$ is too `big', there may be too few $G_\RR$ invariants $\mathcal{J}$ and, if $\mathcal{R}$ is too `small', solvability in terms of an $R\in \mathcal{R}$ may be more difficult.

Although we characterize $\mathcal{T}_{3, \rm PD}\setminus \mathcal{T}_{3, \rm FD}$, one might in principle adapt the method here to work out the necessary and sufficient conditions to characterize the larger set $\mathcal{T}_{3, \rm PD}$ directly. In this case, other parameters should also be involved.  However, by considering the more particular set, we found the calculations and optimizations in choosing an associated $\mathcal{R}$ amenable, and the specific characterization of $\mathcal{T}_{3, \rm PD}\setminus \mathcal{T}_{3, \rm FD}$, of its own interest, succinct.

We comment when $n\geq 4$, if there were known bases, or generating sets of $O(n)$ invariants extant, then one would be able to provide explicit membership criteria for both $\mathcal{T}_{n, \rm FD}$ and $\mathcal{T}_{n, \rm PD}\setminus \mathcal{T}_{n, \rm FD}$. We remark that such generating sets might be identified using Molien's formula and application of the Reynolds operator (cf. Section \ref{calc inv section}).

Finally, we observe that one might consider other sets of tensors, beyond $\mathcal{T}_{n, \rm FD}$ or $\mathcal{T}_{n, \rm PD}$.  For instance, when $n=4$, a tensor $\Gamma$ may not be partially decoupleable in that one axis `splits' off from the others, but say sets of two axes each split off.  In terms of the KPZ system \eqref{eq:1.1}, the equations would separate into two closed systems, each system governing nontrivially at least two components.  The membership problem for such invariant sets of tensors and generalizations is also of interest and left to future investigations.

\subsection{Plan of the paper}
We discuss in Section \ref{two and three components section} explicit characterizations of $\mathcal{T}_{n, \rm FD}$ and $\mathcal{T}_{n, \rm PD}$ when $n=2$, the equality $\mathcal{T}_{2, \rm FD}=\mathcal{T}_{2, \rm PD}$, and characterization of $\mathcal{T}_{2, \rm FD}$ by two direct solutions.  Then, we discuss notions of applied invariant theory in Section \ref{invariants section} that will be useful for our main results; we also discuss here, via notions of covariants, that $\mathcal{T}_{n, \rm PD}\neq \mathcal{T}_{n, \rm FD}$ when $n\geq 3$.   In Section \ref{structure n=2}, we consider the structure of $SO(2)$ and $O(2)$ invariants when $n=2$, give associated integrity bases, useful in the sequel. Then, we consider a more abstract `membership problem' of invariant subsets in Section \ref{membership section}.    The abstract result is applied to characterize $\mathcal{T}_{n, \rm FD}$ in $n\geq 3$ in Section \ref{full decoup sym proof}.  In Section \ref{explicit n=3}, after detailing Olive and Auffray's integrity basis of $O(3)$ invariants on $\mathcal{T}_3$, we consider explicit characterizations of $\mathcal{T}_{3, \rm FD}$ and $\mathcal{T}_{3, \rm PD}\setminus \mathcal{T}_{3, \rm FD}$.

\section{Two and three components systems}
\label{two and three components section}

We first specify and discuss the problem explicitly when $n=2,3$.
Then, we discuss characterization of $\mathcal{T}_{2, \rm FD}$ by two types of `direct' solutions, one by solving the defining equations (Section \ref{prop:2.2 section}), and the other by viewing the problem in terms of differential equations (Section \ref{ODE section}).

\subsubsection{Formulation when $n=2$}
When $n=2$, we display, in terms of $a_0=\Gamma^2_{2,2}, a_1= \Gamma^1_{2,2}, a_2=\Gamma^2_{1,1}, a_3=\Gamma^1_{1,1}$, a trilinear tensor $\Gamma$ as follows:
\begin{equation}
\label{n=2 tensor}\left(\begin{array}{cc}\Gamma^1_{1,1}& \Gamma^1_{1,2}\\
\Gamma^1_{2,1}& \Gamma^1_{2,2}\end{array}\right) = \left(\begin{array}{cc}a_3&a_2\\
a_2&a_1\end{array}\right) \ {\rm and \ } \left(\begin{array}{cc}\Gamma^2_{1,1}& \Gamma^2_{1,2}\\
\Gamma^2_{2,1}& \Gamma^2_{2,2}\end{array}\right) = \left(\begin{array}{cc}a_2&a_1\\
a_1&a_0\end{array}\right).
\end{equation}
Parametrically, we may identify $\Gamma$ by $(a_0, a_1, a_2, a_3)$.  The tensor may also be represented in terms of the function $f$:
\begin{align*}
f(x;\Ga) 
& = \Ga_{11}^1 x_1^3 + 2 \Ga_{12}^1 x_1^2 x_2 + \Ga_{22}^1 x_1 x_2^2 +
\Ga_{11}^2 x_1^2 x_2 + 2 \Ga_{12}^2 x_1x_2^2 + \Ga_{22}^2 x_2^3 \\
& = a_3 x_1^3 +  3a_2x_1^2 x_2  + 3a_1x_1x_2^2
  + a_0 x_2^3.
\end{align*}

Interestingly, when $n=2$, we observe that `partial decoupleability' is the same as `full decoupleability'.

\begin{lem} \label{lem:2.1}
If $\Ga$ is trilinear and $n=2$, partial decoupleability implies full decoupleability, that is 
$\mathcal{T}_{2, \rm PD} = \mathcal{T}_{2, \rm FD}$.
\end{lem}

\begin{proof}
By partial decoupleability, there exists an orthogonal matrix $\si$
and $i$ (we may assume $i=1$) such that
$$
\si\circ\Ga^1 \equiv (\si\circ\Ga)^1_{jk} = \begin{pmatrix}
\beta_1 & 0 \\ 0 & 0 \end{pmatrix},
$$
for some $\beta_1\in \R$.
Then, by the trilinearity of $\si\circ\Ga$ for orthogonal $\si$, we have
$$
\si\circ\Ga^2 \equiv (\si\circ\Ga)^2_{jk} = \begin{pmatrix}
(\si\circ\Ga)^2_{11} & (\si\circ\Ga)^2_{12}  \\ (\si\circ\Ga)^2_{21}  & (\si\circ\Ga)^2_{22} \end{pmatrix} = \begin{pmatrix}0 & 0 \\ 0 & \beta_2 \end{pmatrix},
$$
for some $\beta_2\in \R$. This implies full decoupleability.
\end{proof}

\begin{rem}
\label{rem:1.1}
\rm Although the four entries $a_0, a_1, a_2, a_3$ characterize the $O(2)$ orbit of a trilinear tensor $\Gamma$, one can characterize orbits with less information.  Indeed, with respect to the rotation $\sigma_\theta\in SO(2)\subset O(2)$, for $\theta \in [0,2\pi]$,
\begin{align}
\label{rotation}
\si_\theta = \begin{pmatrix} 
 \cos\theta & - \sin\theta \\ 
 \sin\theta &   \cos\theta
\end{pmatrix},
\end{align}
if $\Ga$ is trilinear, we have
\begin{align*}
 G(\theta) :=& (\si_\theta\circ\Ga)_{22}^1 
=  \Ga_{11}^1 \cos\theta \sin^2\theta -\Ga_{22}^2 \sin\theta\cos^2\theta \\
& \hskip 20mm
+ \Ga_{11}^2(-\sin^3\theta+ 2\cos^2\theta\sin\theta) + \Ga_{22}^1(\cos^3\theta-
2\sin^2\theta \cos\theta).
\end{align*}
Since $G(0)=-G(\pi)$, there is an angle $\theta$ where $(\si_\theta\circ\Ga)_{22}^1 =G(\theta)=0$.  Hence, the orbit of a trilinear $\Gamma$ is characterized by three values or combinations of $a_0,a_1, a_2, a_3$.
The corresponding $f(\sigma_\theta^{-1}x;\Gamma)$ can be put in form
$$f(\sigma_\theta^{-1}x;\Gamma) = \beta_1x_1^3 +\beta_2 x_2^3 + \gamma\big(3x_1^2x_2 - x_2^3\big),$$
where the trace vector $((\si_\theta\circ \Ga)^1_{1,1}, (\si_\theta\circ\Ga)^2_{2,2})= (\beta_1, \beta_2)$,
and $\gamma$ reflects a coupling between $\Gamma^1$ and $\Gamma^2$.

  As the orbit of a fully decoupled trilinear tensor $\Gamma$ is characterized by the values $\Gamma^1_{1,1}=\beta_1, \Gamma^2_{2,2}=\beta_2$, one suspects that a single equation relating $a_0,a_1, a_2, a_3$ would determine if $\Gamma\in \mathcal{T}_{2, \rm FD}=\mathcal{T}_{2, \rm PD}$.  This is the content of Sections \ref{prop:2.2 section} and \ref{ODE section}.
\end{rem}

  \subsubsection{Formulation when $n=3$}

We now consider when $n=3$.  In this case, a general trilinear tensor $\Gamma$ is of the form
\begin{align*}
&  \Ga^1 
= \begin{pmatrix} 
\Ga_{11}^1 & \Ga_{12}^1 & \Ga_{13}^1 \\
\Ga_{21}^1 & \Ga_{22}^1 & \Ga_{23}^1 \\
\Ga_{31}^1 & \Ga_{32}^1 & \Ga_{33}^1
\end{pmatrix}
= \begin{pmatrix} 
a_1 & b_1 & b_3 \\
b_1 & b_2 & c \\
b_3 & c & b_4
\end{pmatrix}  \\
&  \Ga^2
= \begin{pmatrix} 
\Ga_{11}^2 & \Ga_{12}^2 & \Ga_{13}^2 \\
\Ga_{21}^2 & \Ga_{22}^2 & \Ga_{23}^2 \\
\Ga_{31}^2 & \Ga_{32}^2 & \Ga_{33}^2
\end{pmatrix}
= \begin{pmatrix} 
b_1 & b_2 & c \\
b_2 & a_2 & b_5 \\
c & b_5 & b_6
\end{pmatrix}  \\
&  \Ga^3 
= \begin{pmatrix} 
\Ga_{11}^3 & \Ga_{12}^3 & \Ga_{13}^3 \\
\Ga_{21}^3 & \Ga_{22}^3 & \Ga_{23}^3 \\
\Ga_{31}^3 & \Ga_{32}^3 & \Ga_{33}^3
\end{pmatrix}
= \begin{pmatrix} 
b_3 & c & b_4 \\
c & b_5 & b_6 \\
b_4 & b_6 & a_3
\end{pmatrix}  \\
\end{align*}

Full decoupleability is equivalent to the existence of $\sigma\in O(3)$
such that
\begin{align}
\label{full canonical}
\sigma \circ \Ga^1
= \begin{pmatrix} 
\beta_1 & 0 & 0 \\
0 &   0 & 0 \\ 
0 &  0 & 0     
\end{pmatrix} \quad
\sigma\circ \Ga^2
= \begin{pmatrix} 
0 & 0 & 0 \\
0 &  \beta_2 & 0 \\ 
0 &  0 & 0     
\end{pmatrix} \quad
\sigma\circ \Ga^3
= \begin{pmatrix} 
0 & 0 & 0 \\
0 &   0 & 0 \\ 
0 &  0 & \beta_3
\end{pmatrix},
\end{align}
while partial reducibility is equivalent to finding $\sigma\in O(n)$ so that
\begin{align}
\label{partial reduced}
\sigma \circ \Ga^1
= \begin{pmatrix} 
a_3 & a_2 & 0 \\
a_2 &   a_1 & 0 \\ 
0 &  0 & 0     
\end{pmatrix} \quad 
\sigma\circ \Ga^2
= \begin{pmatrix} 
a_2 & a_1 & 0 \\
a_1 &  a_0 & 0 \\ 
0 &  0 & 0     
\end{pmatrix} \quad
\sigma\circ \Ga^3
= \begin{pmatrix} 
0 & 0 & 0 \\
0 &   0 & 0 \\ 
0 &  0 & \beta_3
\end{pmatrix},
\end{align}
which we call the `reduced' forms.  In these cases, $\mathcal{R}$ in Section \ref{discussion section} are the sets of reduced form tensors.

Unlike when $n=2$, however, we comment that $\mathcal{T}_{n, \rm FD}$ is a strict subset of $\mathcal{T}_{n, \rm PD}$ when $n\geq 3$.  One can also characterize the reduced forms of partially but not fully decoupled tensors in terms of the condition in Proposition \ref{prop:2.2}.  See Proposition \ref{prop:PD neq FD} and Lemma \ref{rem:3.2} for these statements.

\begin{rem}
\label{rem:1.2} \rm
When $n=3$, in a general trilinear tensor $\Gamma$ there are $10$ entries.  However, as orbits of $\Gamma$ are invariant under rotation, which may be identified in terms of an axis specified by a unit vector $\vec{n}$ and an angle $\theta$ about this axis, namely $3$ items (two from $\vec{n}$ and one from $\theta$), in effect only $7$ combinations of the $10$ entries characterize the orbit.

Then, to determine orbits of a fully decoupleable tensor, given in terms of $\beta_1, \beta_2, \beta_3$, one suspects $4$ relations.  For a partially decoupleable tensor in reduced form, by rotating the $x_1, x_2$ directions, mirroring the $n=2$ discussion in Remark \ref{rem:1.1}, one sees that its orbit is given by four parameters, say $\beta_1, \beta_2, \beta_3, \gamma$.  One loosely suspects only $3$ algebraically independent relations then should characterize $\mathcal{T}_{3, \rm PD}$.

As we will see, the $n\geq 3$ analysis is more involved than when $n=2$.  We will be able to give necessary and sufficient relations to be in $\mathcal{T}_{n, \rm FD}$ and when $n=3$ to be in $\mathcal{T}_{3, \rm PD}\setminus \mathcal{T}_{3, \rm FD}$.  However, the number of explicit characterizing relations that we will find in Section \ref{explicit n=3}, by the use of applied invariant theory, when $n=3$ is larger, reflecting some algebraic dependencies among the invariants used.
\end{rem}

\subsection{Direct solution when $n=2$}
\label{prop:2.2 section}

Recall, in terms of $a_0, a_1, a_2, a_3$, the representation of the tensor $\Gamma$ in \eqref{n=2 tensor}.

\begin{prop} \label{prop:2.2}
Assume that $\Ga$ is trilinear and $n=2$.
Then, the KPZ-system \eqref{eq:1.1} is fully decoupleable if and only if
the relation
\begin{equation} \label{eq:2.1}
a_2(a_2-a_0)=a_1(a_3-a_1)
\end{equation}
holds.
\end{prop}

\begin{proof}
Consider the rotation $\si_\theta$ given in \eqref{rotation}.
Since $\Ga$ is trilinear, recall and define
\begin{align*}
 F(\theta) := & (\si_\theta\circ\Ga)_{12}^1 
=  \Ga_{11}^1 \cos^2\theta \sin\theta +\Ga_{22}^2 \sin^2\theta\cos\theta \\
& \hskip 20mm
+ \Ga_{11}^2(\cos^3\theta - 2 \cos\theta\sin^2\theta)
+ \Ga_{22}^1 (\sin^3\theta -2 \sin\theta\cos^2\theta)\\
 G(\theta) :=& (\si_\theta\circ\Ga)_{22}^1 
=  \Ga_{11}^1 \cos\theta \sin^2\theta -\Ga_{22}^2 \sin\theta\cos^2\theta \\
& \hskip 20mm
+ \Ga_{11}^2(-\sin^3\theta+ 2\cos^2\theta\sin\theta) + \Ga_{22}^1(\cos^3\theta-
2\sin^2\theta \cos\theta).
\end{align*}
One seeks a condition for $\Ga$ such that $F(\theta)=G(\theta)=0$ holds for some 
common $\theta$.  

We comment that if we would use the orthogonal matrix 
$\begin{pmatrix} 0 & 1 \\ 1 & 0 \end{pmatrix} \si_\theta$ with determinant $-1$, then these equations would hold with respect to `reversed' coefficients
$\Ga^1_{11}=a_0, \Ga^2_{11}=a_1, \Ga^1_{22}=a_2, \Ga^2_{22}=a_3$.
In the following, we will continue with the first type of orthogonal matrix $\si_\theta$.  

First, assume $F=G=0$  holds at $\cos\theta=0$.  Then, since $F(\theta)=\Ga_{22}^1 \sin^3
\theta$, $G(\theta)=-\Ga_{11}^2 \sin^3\theta$, $F=G=0$ holds only
if $\Ga_{22}^1=\Ga_{11}^2 =0$.

Next, assume $\cos\theta\not=0$ and divide $F, G$ by $\cos^3\theta$.
Then, setting $x= \tan \theta\in \R$, we see $F=G=0$ is equivalent to
two equations:
\begin{align*}
& \Ga_{22}^1 x^3 + (\Ga_{22}^2-2\Ga_{11}^2) x^2 + (\Ga_{11}^1-2\Ga_{22}^1) x + \Ga_{11}^2=0, \\
& \Ga_{11}^2 x^3 + (2 \Ga_{22}^1-\Ga_{11}^1)x^2 
+ (\Ga_{22}^2-2\Ga_{11}^2) x-\Ga_{22}^1=0.
\end{align*}
Or, writing $a_0, a_1, a_2, a_3$ as in the statement of the proposition, we have
\begin{align}  \tag{1}
&a_1 x^3 + (a_0-2a_2) x^2 + (a_3-2a_1) x + a_2=0, \\
& a_2 x^3 + (2 a_1-a_3)x^2 + (a_0-2a_2) x-a_1=0.  \tag{2}
\end{align}
One looks for a condition for $\Ga$ such that these two equations have
a common root $x$.

Assume $x$ for $x\not=0$ is a real solution of (1) (respectively (2)).
[Note that $x=0$ is a solution of both (1) and (2) only when $a_1=a_2=0$, 
that is $\Ga_{22}^1=\Ga_{11}^2=0$, same as above.] 
Then, one can observe that 
$X:=-\frac1x$ satisfies (2) (respectively\ (1)).  Thus, if $(x,y,z)$ are three real solutions of (1),
then $(-\frac1x,-\frac1y,-\frac1z)$ are three real solutions of (2).  A common solution
$x=-\frac1x$ cannot happen, since $x$ cannot be real in this case.  
Therefore, a common 
solution exists if $x=-\frac1y$ or $x=-\frac1z$ or 
$y=-\frac1x$ or $y=-\frac1z$ or $z=-\frac1x$ or $z=-\frac1y$.
Namely, for (1) and (2) to have a common solution, two solutions of (1)
must satisfy $xy=-1$.  The converse is also true.

The problem is reduced to find a condition for (1) to have two real solutions $\a, \b$
satisfying $\a\b=-1$.  Let $(\a,\b,\ga)$ be three solutions of (1).  Then,
since $a_1(x-\a)(x-\b)(x-\ga)= a_1 \big(x^3 - (\a+\b+\ga)x^2 + (\a\b+\b\ga + \ga\a)x 
- \a\b\ga\big)$, setting $A=\a+\b = \a - \frac1\a \in \R$ (note: conversely for any given $A\in \R$, 
$A=\a - \frac1\a$, that is $\a^2-A\a-1=0$ has two real solutions, since $D:= A^2+4>0$),
we see that
\begin{align*}
-a_1 (A+\ga) = a_0-2a_2, \quad
a_1(-1+A \ga)= a_3-2a_1, \quad
a_1 \ga = a_2.
\end{align*}
Substituting the last to the first and the second, we have
\begin{align*}
-a_1 A-a_2 = a_0-2a_2, \quad
-a_1 +a_2 A= a_3-2a_1,
\end{align*}
that is 
$a_1 A= a_2 - a_0$ and $a_2 A= a_3-a_1$.
Therefore $A$ (and therefore $\a$) exists if and only if 
$a_2(a_2-a_0)=a_1(a_3-a_1)$ is satisfied.

More precisely, this statement is correct  if $a_1\not=0$ (since $A=(a_2-a_0)/a_1$ from
the first and this $A$ satisfies the second).  In the case $a_1=0$, we have $a_2-a_0=0$
from the first.  But this satisfies the relation $a_2(a_2-a_0)=a_1(a_3-a_1)$.  Moreover,
in the case $a_1=a_2=0$ which we excluded above, this relation holds. 
Summarizing these observations, we have shown that two equations 
(1) and (2) have a common real root if and only if the condition \eqref{eq:2.1} 
holds for $a_0, a_1, a_2, a_3$.  This concludes the proof of the proposition.
\end{proof}

\subsection{Solution by ODE when $n=2$}
\label{ODE section}

Our aim will be to rederive the
condition in Proposition \ref{prop:2.2} by use of a certain ODE which allows also to deduce all fully decoupled tensors on the $SO(2)$ and $O(2)$ orbits of $\Ga$, that is those in form $\sigma\circ \Gamma$ for $\sigma$ in $SO(2)$ and $O(2)$ (cf. Section \ref{invariants section}).  We will restrict in the following to $\Ga\neq 0$, as the claim of Proposition \ref{prop:2.2} is evident when $\Ga=0$.

\begin{prop} \label{prop:ode} Assume that $\Ga\neq 0$ is trilinear and $n=2$. Then,  $\Gamma$
is fully decoupleable if and only if the relation $a_2(a_2-a_0)=a_1(a_3-a_1)$  given in \eqref{eq:2.1} holds.   

Moreover, when \eqref{eq:2.1} holds, with respect to the rotation $\sigma_\theta$ defined in \eqref{rotation} and a reflection $\mathcal{N}$ defined by $\mathcal{N}(x_1,x_2) = (x_2,x_1)$, we have $\sigma_{\theta} \circ \Ga = R$ and $\sigma_{\psi} \mathcal{N} \circ \Ga = S$ are in fully decoupled reduced form, where
\begin{align*}
e^{4i\theta} & = \frac{\left(a_0+a_2\right) + i \left(a_1+a_3\right)}{\left(a_0-3a_2\right) + i \left(a_3-3a_1\right)},  & R_{11}^1 +i R_{22}^2  & = e^{-i\theta} \left[\left(a_0+a_2\right) + i \left(a_1+a_3\right)\right], \\
e^{4i\psi} & = \frac{\left(a_3+a_1\right) + i \left(a_2+a_0\right)}{\left(a_3-3a_1\right) + i \left(a_0-3a_2\right)}, & S_{11}^1 + i S_{22}^2  & = e^{-i\psi} \left[\left(a_3+a_1\right) + i \left(a_0+a_2\right)\right].
\end{align*}
These $\sigma_\theta$ with four possible $\theta$, and $\sigma_\psi\mathcal{N}$ with four possible $\psi$ are the only possible maps in $O(2)$ from $\Gamma$ to the reduced form $R$.
\end{prop}

\begin{proof}
Differentiating $\sigma_\theta$ as defined in \eqref{rotation} we get
\begin{align*}
\partial_\theta \si_\theta = \left(\begin{array}{rl} -\sin \theta& -\cos \theta\\
\cos \theta &-\sin \theta\end{array}\right) = \left(\begin{array}{rl} 0& -1\\
1 &0\end{array}\right)\si_\theta.
\end{align*}
Then, with $\sigma = \si_\theta$, and $\epsilon_{11}=\epsilon_{22}=0$ and $\epsilon_{21} = -\epsilon_{12}=1$, it follows from \eqref{eq:1.3} that
\begin{align*}
\partial_\theta (\sigma_\theta \circ \Gamma)^i_{jk}
&= \sum_{i', j', k'}\sum_{m=1,2} \epsilon_{im}\sigma_{mi'}\sigma_{jj'}\sigma_{kk'}\Gamma^{i'}_{j'k'}\\
&\ + \sum_{i', j', k'}\sum_{m=1,2} \epsilon_{jm}\sigma_{ii'}\sigma_{mj'}\sigma_{kk'}\Gamma^{i'}_{j'k'}+\sum_{i', j', k'}\sum_{m=1,2} \epsilon_{km}\sigma_{ii'}\sigma_{jj'}\sigma_{mk'}\Gamma^{i'}_{j'k'}.
\end{align*}

With our convention $a_0= (\sigma_\theta\circ \Gamma)^{2}_{22}$, $a_1= (\sigma_\theta\circ\Gamma)^1_{22}$, $a_2 = (\sigma_\theta\circ\Gamma)^2_{12}$ and $a_3 = (\sigma_\theta\circ \Gamma)^{1}_{11}$,
we have, after a calculation,
\begin{align*}
\partial_\theta (a_0, a_1, a_2, a_3)^t
=(3a_1, 2a_2-a_0, a_3-2a_1, -3a_2)^t
 = L (a_0, a_1, a_2, a_3)^t,
\end{align*}
where $L$ and its left eigenvector matrix $E$ are given by
$$L = \left(\begin{array}{cccc}
0&3&0&0\\
-1&0&2&0\\
0&-2&0&1\\
0&0&-3&0\end{array}\right), \ \ {\rm and \ \ }
E=\left(\begin{array}{cccc}
 1 & -3 i & -3 & i \\
 -1 & -3 i & 3 & i \\
 -1 & i & -1 & i \\
 1 & i & 1 & i 
\end{array}\right).$$
Here, the rows of $E$ are left eigenvectors $u_1, u_2, u_3, u_4$ with corresponding eigenvalues $\lambda_1=3i, \lambda_2=-3i, \lambda_3=i, \lambda_4=-i$; also the superscript `$t$' indicates transpose.

Since $\Ga \neq 0$, $\|a\|^2 := a_0^2 + a_1^2 + a_2^2 + a_3^2 > 0$.  Define 
$E(a_0, a_1, a_2, a_3)^t
=: v(0)$ with the natural extension to $v(\theta)$ defined in terms of the entries of $\sigma_\theta \circ \Ga$.
 Since $E$ is invertible, $\|a\|^2 > 0$ implies $\|v(0)\|^2 > 0$. The evolution of $v(\theta)$ under rotations is given by
\begin{align}
\label{rot form}
\partial_\theta v = EL\left(\begin{array}{c}a_0\\a_1\\a_2\\a_3\end{array}\right) = \left(\begin{array}{cccc}-i&0&0&0\\0&i&0&0\\0&0&-3i&0\\0&0&0&3i\end{array}\right) E\left(\begin{array}{c}a_0\\a_1\\a_2\\a_3\end{array}\right) = \left(\begin{array}{cccc} -i&0&0&0\\0&i&0&0\\0&0&-3i&0\\0&0&0&3i\end{array}\right) v.
\end{align}
Hence,
$v_j(\theta) = e^{\lambda_j\theta}v_j(0)$ for $j=1,2,3,4$.
The tensor
$\Gamma$ can be rotated to a fully decoupled system with parameters $(\beta_1, 0, 0, \beta_2)$, 
if and only if the following system is consistent and can be solved for real parameters $\theta, \beta_1$ and $\beta_2$: 
\begin{align}
\label{solving help}
v_1(\theta)= e^{3i\theta}v_1(0) = \beta_1+i\beta_2, \quad  v_2(\theta)= e^{-3i\theta}v_2(0) = \beta_1-i\beta_2, \nonumber\\
v_3(\theta) = e^{i\theta}v_3(0) = \beta_1 - i\beta_2\ \ {\rm and \ \ }
v_4(\theta)= e^{-i\theta}v_4(0) = \beta_1+i\beta_2.
\end{align}
Assuming consistency for this system, $4(\beta_1^2 + \beta_2^2) = \sum |v_i(\theta)|^2 = \sum |v(0)|^2 = \|a\|^2 > 0$. Consequently, none of the denominators in the following expressions vanish and we obtain
$$\frac{v_4(\theta)}{v_1(\theta)}= \frac{\beta_1 + i\beta_2}{\beta_1 +i\beta_2} = 1 = e^{-4i\theta} \frac{v_4(0)}{v_1(0)} = e^{-4i\theta} \left(\frac{\left(a_0+a_2\right) + i \left(a_1+a_3\right)}{\left(a_0-3a_2\right) + i \left(-3a_1+a_3\right)}\right).$$
An angle $\theta$ satisfying this condition exists if and only if 
$\left|\frac{v_4(0)}{v_1(0)}\right| = 1$.
Squaring,
we get the necessary condition
$$(a_0-3a_2)^2 + (a_3-3a_1)^2 = (a_0+a_2)^2 + (a_1+a_3)^2$$
which after algebra reduces to the desired condition
$a_2(a_2-a_0) = a_1(a_3-a_1)$.

 Contingent upon this condition, and the standing assumption $\|a\|^2>0$,  we obtain
$$
e^{4i\theta} = \frac{\left(a_0+a_2\right) + i \left(a_1+a_3\right)}{\left(a_0-3a_2\right) + i \left(-3a_1+a_3\right)}, \qquad \beta_1+i\beta_2 = e^{-i\theta} \left[\left(a_0+a_2\right) + i \left(a_1+a_3\right)\right].
$$
showing sufficiency as we have identified the rotation angle $\theta$ and the equivalent decoupled tensor $R$ with parameters $(\beta_1, 0, 0, \beta_2)$. Applying the same argument starting with the tensor $\mathcal{N} \circ \Ga$, which is parameterized by $(a_3,a_2,a_1,a_0)$, finishes the proof.
\end{proof}

We comment that this procedure identifies four possible rotations $\theta \in [0,2\pi)$, and correspondingly, $4$ different possibilities for $\beta_1$ and $\beta_2$. These solutions are related in that, if we pick one of the possible angles as $\theta$, the set of allowed rotations is given by $R_{\theta + m \pi/2}$ for $m = 0,1,2,3$  yielding the fully decoupled tensors represented parametrically as $(\beta_1, 0, 0, \beta_2), (\beta_2, 0, 0, -\beta_1), (-\beta_1, 0, 0, -\beta_2)$ and $(-\beta_2, 0, 0, \beta_1)$ respectively. These tensors are all on the $SO(2)$ orbit of the tensor with parameters $(a_0,a_1,a_2,a_3)$ provided that $a_2(a_2-a_0) = a_1(a_3-a_1)$.

To characterize all the decoupled tensors on the $O(2)$ orbit of $\Ga \in  \mathcal{T}_{2, \rm FD}$, we note that reflections in $O(2)$ are obtained by compositions $\sigma_\psi \circ \mathcal{N}, \psi \in [0, 2 \pi)$. Since $\sigma_\psi \mathcal{N} = \mathcal{N} \sigma_{-\psi}$, we get $\sigma_\psi \mathcal{N} \circ \Ga = \mathcal{N} \circ (\sigma_{-\psi} \circ \Ga)$. This tensor is fully decoupled only if $\psi = - \theta$ yielding $\sigma_\psi \mathcal{N} \circ \Ga = \mathcal{N} \circ R$. This gives $4$ decoupled tensors on the $O(2)$ orbit, arising from reflections, represented parametrically by $(\beta_2, 0, 0, \beta_1), (\beta_1, 0, 0, -\beta_2), (-\beta_2, 0, 0, -\beta_1)$ and $(-\beta_1, 0, 0, \beta_2)$. Consequently, except when $\beta_1 = \pm \beta_2$ or $\beta_1 \beta_2 = 0$, there are exactly $8$ distinct decoupled tensors on the $O(2)$ orbit of $\Ga \in \mathcal{T}_{2, \rm FD}$. 

We also note that this `ODE' proof makes implicit use that the norms of $\{v_i(\theta)\}_{i=1,2,3,4}$ do not depend on $\theta$.  These norms are examples of `invariants', more discussed in the next section.

\section{Invariants}
\label{invariants section}

We review some of the basic notions of `applied invariant theory', tailored to our context.  In the following, $G=O(n)$ or $G=SO(n)$.  Mostly, we will focus on $G=O(n)$ in the sequel, although there will be occasions when considering $G=SO(n)$ will be of use.
\begin{defn}
An $\R$-valued function $I(\Ga)$ on trilinear tensors $\mathcal{T}_n$ is called an invariant under the
action of $G$
if $I(\Ga) = I(\si\circ\Ga)$ holds for every $\si\in G$.

\end{defn}
In particular, an invariant $I$ under the action of $G$ is constant along $G$-orbits in $\mathcal{T}_n$.

\begin{defn}
A polynomial $J:\R^n\times \mathcal{T}_n \rightarrow \R$, homogeneous separately in both arguments, is called a covariant under the action of $G$ if $J(x, \Gamma) = J(\sigma x,\sigma\circ \Gamma)$ for every $\sigma\in G$.
\end{defn}

Note that every invariant is also a covariant.  
The homogeneous third degree polynomial, associated to $\Gamma\in \mathcal{T}_n$, 
$f(x;\Ga) =\sum_{i,j,k=1,\ldots, n}\Gamma^i_{j,k}x_ix_jx_k$,
is a covariant by Lemma \ref{lem:4.1}.

\begin{defn}
\label{vector cov defn}
We will say that a vector valued function $v: \mathcal{T}_n \to \R^n$, respectively a real symmetric matrix valued function $Q: \mathcal{T}_n \to M_{n \times n}^{\text{sym}}$, is a covariant, if the linear form $p(\Ga,x) = v(\Ga)^t x$, respectively the quadratic form $q(\Ga,x) = x^t Q(\Ga) x$, is a covariant function.  
We will also say a tensor valued function $D: \mathcal{T}_n \to \mathcal{T}_n$ is covariant if $D(\sigma\circ \Ga) = \si \circ D(\Ga)$ for all $\sigma \in G, \Ga \in \mathcal{T}_n$.
\end{defn}

For instance, as $\Delta=\sum_i\partial^2_{x_i}$ is $O(n)$ invariant, the function $\Delta f(x;\Ga)$ is $O(n)$ covariant, since $\Delta f(\sigma x;\sigma\circ \Ga)= \Delta \big[f(\sigma x; \sigma\circ \Ga)\big]=\Delta \big[f(x; \Ga)\big]=\Delta f(x; \Ga)$.  The matrix $\Gamma^{*2}$ with entries 
\begin{equation}
\label{Gamma *2}
\Gamma^{*2}_{k,\ell} = \sum_{i,j}\Gamma^i_{j,k}\Gamma^i_{j,\ell}
\end{equation}
is also $O(n)$ covariant.
Indeed, $(\sigma x)^t (\sigma\circ\Gamma)^{*2}(\sigma x) = \sum_{k,\ell}\sum_{i,j}(\sigma x)_k(\sigma x)_\ell (\sigma\circ\Gamma)^i_{j,k}(\sigma\circ\Gamma)^i_{j, \ell}$.  One sees $\sum_k (\sigma x)_k (\sigma\circ \Ga)^i_{j,k} = \sum_{i', j', k'} \sigma_{i,i'}\sigma_{j, j'}\Gamma^{i'}_{j', k'}x_{k'}$ using that $\sum_k \sigma_{k,s}\sigma_{k,k'} = 1(s=k')$, given $\sigma\in O(n)$.  Substituting in, using that $\sigma\in O(n)$ again, we conclude $(\sigma x)^t (\sigma\circ\Gamma)^{*2}(\sigma x) = x^t\Gamma^{*2}x$ and $(\sigma \circ \Ga)^{*2} = \sigma \Ga^{*2} \sigma^t$.
Since $\Ga^{*2}$ and $(\sigma \circ \Ga)^{*2} = \sigma \Ga^{*2} \sigma^t$ have the same characteristic equation, symmetric functions of the eigenvalues of $\Ga^{*2}$, 
including its trace and determinant, are $O(n)$ invariant 
 (as will be used in Section
 \ref{full decoup sym proof}).

Consider now the polynomial invariant functions on $\mathcal{T}_n$, denoted $\R[\mathcal{T}_n]^G$, which form a sub-algebra in the space of polynomials $\R[\mathcal{T}_n]$. Since $G$ is a compact Lie group, it is well known that the sub-algebra $\R[\mathcal{T}_n]^G$ is finitely generated (Chapter 2 in \cite{Sturmfels}).  
We will call a collection of polynomial invariants which generate $\R[\mathcal{T}_n]^G$ as a `generating set', which also may be known as a `basis'.  That is, $\{f_1(\Gamma), \ldots, f_\alpha(\Gamma)\}\subset \R[\mathcal{T}_n]^G$ is called a generating set if any polynomial invariant $I$ can be written as a polynomial in $\{f_1, \ldots, f_\alpha\}$. 
 A generating set or basis of $\R[\mathcal{T}_n]^G$ is sometimes called an `integrity basis' as in \cite{OA}.  We note that such an integrity basis may possess `syzygies', or dependent relations between members (cf. Section 1.3 in \cite{Sturmfels}).

An important property of generating sets, which always contain an integrity basis, is that they `separate' or are in $1:1$ correspondence with $G$ orbits.  Note on an $G$ orbit in $\mathcal{T}_n$, values of the members of a generating set, as they are invariants, are constant.

\begin{prop}
\label{prop:separating}
Let $\mathcal{A}$ be a generating set of $\R[\mathcal{T}_n]^G$.  Then, on two different $G$ orbits in $\mathcal{T}_n$, the values of $\mathcal{A}$ are different.
\end{prop}
We refer to Appendix C of \cite{AS} for a proof of this proposition.

One may count the cardinality of an integrity basis, and its specification in terms of degrees and any `syzygies' (relations between members), via `Molien's formula' and Hilbert series.
Let $\mathcal{L}=\mathcal{L}_g:G \rightarrow GL(V)$ be a linear representation of the compact group $G$ on a vector space $V$ over $\R$.  Define
$$\Phi(\lambda) = \frac{1}{|G|}\int \frac{1}{\det(\mathds{1} - \lambda \mathcal{L}_g)} dg,$$
where $dg/|G|$ is the associated Haar probability measure and $|G|$ is the volume of $G$.  This average of the reciprocal of a characteristic polynomial may be expanded to give the associated Hilbert series:  Let $\R[V]_d^G$ be the vector space of $V$-polynomial invariants of degree $d$.  Then, it is known that
$$\Phi(\lambda) = \sum_{d\geq 0}c_d\lambda^d$$
is a generating function where $c_d$ is the number of linearly independent polynomial invariants of degree $d$ in $\R[V]_d^G$ (cf. Chapter 2 in \cite{Sturmfels}).

These formulas will play a role in Sections \ref{structure n=2} and \ref{explicit partial n=3}.  For example, a case where we will apply the formula is when $n=2$, $G=SO(2)$ is parametrized by an angle $\theta\in [0, 2\pi]$ with volume $|G|=2\pi$, $V=\mathcal{T}_2$ is the space of $2\times 2\times 2$ trilinear tensors, parametrized by four parameters $a_0, a_1, a_2, a_3$, and $\mathcal{L}_g$ is a $4\times 4$ matrix representing the group action (rotation by $g=\theta$) on the tensor space $V=\mathcal{T}_2$.

In general, it is a difficult question to find an integrity basis, although there are standard procedures involving Reynolds operators that will yield information.

However, in the case $n=3$ with respect to $O(3)$ action, we will use the integrity basis found by Olive and Auffray \cite{OA} to characterize fully decoupled and partially decoupled tensors in Sections \ref{full decoup sym proof}, \ref{explicit n=3}.  In the case $n=2$, we will compute a generating set, or integrity basis, with the aid of Molien's formula, in Section \ref{structure n=2}.

\subsection{Calculating invariants}
\label{calc inv section}

We discuss some ways to find invariants $I$, useful in the sequel.  
A classic method is by the well known `Reynolds operator', that is by averaging a 
polynomial function $p$ along group actions of $\Gamma$ (cf. Chapter 2 in \cite{Sturmfels}):
Recall $G$ is either $SO(n)$ or $O(n)$.  For every polynomial $p$ of the entries of $\Gamma\in \mathcal{T}_n$, we have
\begin{align}
\label{lem:reynolds}
I(\Ga) := \frac{1}{|G|}\int_{G}p\big(\sigma\circ \Gamma\big)d\sigma
\end{align}
is a $G$-invariant, where $d\sigma/|G|$ is Haar probability measure and $|G|$ is the volume of $G$.

\begin{rem}
\rm One may in principle use the Reynolds operator \eqref{lem:reynolds} with respect to a list of minimal polynomials $p$ of the parameters of $\Gamma$ to exhaustively compute invariants, for any $n\geq 2$.  

These can be subsequently organized with respect to Molien's formula.  
Indeed, 
one can compute the invariants $I_p$ from a list of polynomials $p$.
A generating collection, which would include an integrity basis, could be identified in the process.
\end{rem}

We will however primarily make use of a decomposition of $\Gamma$ into a traceless tensor and a rank $1$ tensor to compute invariants; see \cite{Weyl} for related methods.  Indeed, in Section \ref{what are inv n=2}, we compute them for $n=2$, and discuss their use in the computations of Olive and Auffray when $n=3$ in Section \ref{OA section}.

Consider the `trace' vector of $\Ga$ in $\R^n$ defined by
\begin{align}
\label{u eqn}
u= {\rm Trace}\; \Ga = \frac{1}{6}\nabla\big(\Delta f(x;\Ga)\big) = \Big(\sum_{\ell=1}^n \Gamma^1_{\ell, \ell}, \ldots, \sum_{\ell=1}^n \Gamma^n_{\ell, \ell}\Big ).
\end{align}
One may see that $u$ is an $O(n)$ covariant vector:  Indeed, $(\sigma x)\cdot {\rm Trace}\; (\sigma \circ \Ga) = x\cdot {\rm Trace}\; \Ga$ using $\sum_i \sigma_{i, k} \sigma_{i, \ell} = 1(k=\ell)$.  Moreover, similarly,  
$I(\Ga) = \|u\|^2=u\cdot u$ is an $O(n)$ invariant.

To find other invariants, given a trace vector $u={\rm Trace}\; \Ga \in \R^n$, we can form a homogeneous cubic polynomial 
$$f_1(x;\Ga) = \frac{3}{n+2}(u\cdot x)\|x\|^2,$$
corresponding to a tensor with the same trace vector:  Note on $\R^n$ that
$$\Delta(x_i \|x\|^2) = \Delta\big(x_i^3 + \sum_{j\neq i} x_ix_j^2\big) = (2n+4)x_i,$$
and therefore
$$\frac{1}{6}\nabla\big(\Delta f_1(x)\big)= \frac{1}{6} \frac{3}{n+2} \nabla\big(\Delta (u\cdot x)\|x\|^2\big) = u.$$
Since $u\cdot x$ is $O(n)$ covariant and $\|\sigma x\|^2=\|x\|^2$, the function $f_1$ is a $O(n)$ covariant.
It corresponds to the tensor $\mathcal{B}$ given by
$$\mathcal{B}(X,Y,Z) = \frac{1}{n+2}\Big[(u\cdot X)(Y\cdot Z) + (u\cdot Y)(Z\cdot X) + (u\cdot Z)(X\cdot Y)\Big],$$
which is also $O(n)$ covariant, as can be seen by its form.

With respect to a general tensor $\Gamma$, we may decompose $f$ as $f(x) = f_1(x) + \frac{1}{n+2}f_3(x)$ where
$$f_3(x) = (n+2)\big(f(x)-f_1(x)\big).$$
Observe that the trace vector $\frac{1}{6}\nabla\big( \Delta f_3\big)$ with respect to the homogeneous cubic polynomial $f_3$ is the zero vector by construction.
Here, $f_3$ is also covariant, corresponding to $O(n)$ covariant tensor $\D = (n+2)\big(\Gamma - \mathcal{B}\big)$.

One can find other $O(n)$ invariants and covariants by combining $u$, $\mathcal{B}$, $\D$ in various ways.  For instance, the trace and determinant of the $n\times n$ matrix $\D^{*2}$ and a vector $w$, defined in terms of $\D$ and $u$, with entries 
\begin{align}
\label{eq:3.5}
\D^{*2}_{k,\ell} = \sum_{i,j=1,\ldots, n}\D^i_{j,k}\D^i_{j,\ell}, \ \  w_m = \sum_{i,j=1,\ldots, n} \D^i_{j,m}u_i u_j
\end{align}
can be seen to be invariants and covariant with respect to $O(n)$.  These will also be of use in specifying Olive and Auffray's basis in Section \ref{OA section}.

\subsection{$\mathcal{T}_{n, \rm FD} \neq \mathcal{T}_{n, \rm PD}$ when $n\geq 3$}

We show that there are partially decoupleable tensors which are not fully decoupleable when $n\geq 3$, using notions of covariants defined earlier.  Also, when $n=3$, we state when a $\Gamma\in \mathcal{T}_{3, \rm PD}$ is not in $\mathcal{T}_{3, \rm FD}$.  

The following proposition complements the equality of the sets shown in Proposition \ref{prop:2.2}
when $n=2$.

\begin{prop}
\label{prop:PD neq FD}
The inclusion $\mathcal{T}_{n, \rm FD}\subset \mathcal{T}_{n, \rm PD}$ is strict when $n\geq 3$.
\end{prop}

\begin{proof}
Recall $\Gamma^{*2}$ and $u = {\rm Trace}\, \Gamma$ defined in \eqref{Gamma *2}, \eqref{u eqn}.

  We claim that the $n\times n$ matrices $\Gamma^{*2}$ with entries $\Gamma^{*2}_{k,\ell}=\sum_{i, j}\Gamma^i_{j, k}\Gamma^i_{j, \ell}$ and the inner product $u\cdot \Gamma$ with entries $(u\cdot \Gamma)_{k, \ell} = \sum_i \Gamma^i_{k, \ell}u_i$ are equal when $\Gamma$ is fully decoupleable.  Indeed, the equality follows by evaluating the $O(n)$ covariants $\Gamma^{*2}$ and $u\cdot \Gamma$ on the canonical form when $\Gamma^i_{i,i}=\beta_i$ and $\Gamma^i_{j, k}=0$ when $j\neq i$ or $k\neq i$.  In this case, $u = (\beta_1, \ldots, \beta_n)$, and $u\cdot\Gamma = \Gamma^{*2}$ is a diagonal matrix with diagonal elements $\beta_1^2, \ldots, \beta_n^2$.

However, one may find forms $\Gamma$ of partially decoupled tensors where $\Gamma^{*2}\neq u\cdot \Gamma$ when $n\geq 3$.  
Indeed, consider $\Gamma$ in a reduced form (cf.
\eqref{partial reduced} when $n=3$), 
where its $2\times 2\times 2$ subtensor $\Gamma^1_{j,k} = G^1_{j,k}, \Gamma^2_{j,k}=G^2_{j,k}$ for $j,k = 1, 2$ is in form
\begin{align}
\label{n=2 canonical}
\left(\begin{array}{cc} \beta_1&\gamma\\\gamma&0\end{array}\right), \ \ \left(\begin{array}{cc}\gamma&0\\0&\beta_2-\gamma\end{array}\right),
\end{align}
and $\Gamma^i_{i,i}=\beta_i$ for $i=3, \ldots, n$ and $\Gamma^i_{j,k} = 0$ when $i\geq 3$ and $j\neq i$ or $k\neq i$.
Clearly the $n$th component is `split off' and $\Gamma$ is partially decoupled.

We compute $\Gamma^{*2}_{1,1} = 2\gamma^2 + \beta_1^2$ and $(u\cdot \Gamma)_{1,1}=\gamma\beta_2 + \beta_1^2$.  Clearly, when $\gamma\beta_2\neq 2\gamma^2$, which is the condition that the subtensor $G$ is not fully decoupleable (cf. Proposition \ref{prop:2.2}), the two matrices do not agree.
\end{proof}

\begin{lem}
\label{rem:3.2}
When $n=3$, if a tensor $\Gamma$ is partially decoupleable, then it is not fully decoupleable
exactly when its reduced form \eqref{partial reduced}
is such that $a_2(a_2-a_0)\neq a_1(a_3-a_1)$. 
\end{lem}

\begin{proof} Consider the reduced form \eqref{partial reduced}.  On the one hand, if $a_2(a_2-a_0)=a_1(a_3-a_1)$, by Proposition \ref{prop:2.2}, we would have by a rotation of axes $x_1, x_2$ that $\Gamma$ can be put in reduced fully decoupleable form.  

On the other hand, given a reduced form $\Gamma\in \mathcal{T}_{3, \rm PD}$ such that $a_2(a_2-a_0)\neq a_1(a_3-a_1)$, rotate its $2\times 2\times 2$ subtensor (cf. Remark \ref{rem:1.1}) so that it is in form \eqref{n=2 canonical}.  If $\beta_2\neq 2\gamma$, by the proof of Proposition \ref{prop:PD neq FD}, $\Gamma\not\in \mathcal{T}_{3, \rm FD}$.  However, $\beta_2\neq 2\gamma$ is the Proposition \ref{prop:2.2} condition, equivalent to the unrotated $a_2(a_2 - a_0)\neq a_1(a_3-a_1)$, for the subtensor not to be fully decoupleable.  
\end{proof}

\section{The structure of invariants when $n=2$}
\label{structure n=2}
We will find an integrity basis when $n=2$ by computing invariants via the decomposition of $\Gamma$ mentioned in Section \ref{calc inv section}.  Some of these calculations will be useful in the sequel.

To identify an integrity basis of $O(2)$ invariants, we first consider Molien's formula.  Recall that $O(2)$ is a semidirect sum of $SO(2)$ and $\Z_2$, corresponding to `rotations with a reflection', that is $g\in O(2)$ when $g=(\theta, \epsilon)$ for $\theta\in [0,2\pi]$ and $\epsilon\in \{0,1\}$, say.

The action of $\sigma\in SO(2)$ on $\Gamma$ may be given via a linear representation, a $4\times 4$ matrix $\mathcal{L}_\sigma$ acting on the coefficients $\Gamma\sim(a_0, a_1, a_2, a_3)$ (thought of as a column vector).  One way to find $\mathcal{L}_\sigma$ is to compute $\sigma\circ \Gamma$ from \eqref{eq:1.3}. Another way is to calculate the coefficients of the transformed function $f(x;\sigma\circ \Gamma)=f(\sigma^{-1}x;\Ga)$.  With respect to $\sigma=\sigma_\theta$ in \eqref{rotation}, we find

$$
\mathcal{L}_\sigma = \left(
\begin{array}{cccc}
 \cos ^3\theta & 3 \sin \theta \cos ^2\theta & 3 \sin ^2\theta \cos \theta
    & \sin ^3\theta \\
 -\sin \theta \cos ^2\theta & \cos ^3\theta-2 \sin ^2\theta
   \cos \theta & 2 \sin \theta \cos ^2\theta-\sin ^3\theta & \sin ^2\theta
   \cos \theta \\
 \sin ^2\theta \cos \theta & \sin ^3\theta-2 \sin \theta \cos ^2\theta &
   \cos ^3\theta-2 \sin ^2\theta \cos \theta & \sin \theta \cos ^2\theta
   \\
 -\sin ^3\theta & 3 \sin ^2\theta \cos \theta & -3 \sin \theta \cos
   ^2\theta & \cos ^3\theta \\
\end{array}
\right)
$$

We now use Molien's formula to compute the Hilbert's series:  Write
\begin{align*}
\det(\mathds{1}_{2\times 2} - \lambda \mathcal{L}_\sigma) &=1+2 \lambda ^2+\lambda ^4-2 \left(\lambda +\lambda ^3\right) \cos\theta+2 \lambda ^2 \cos(2 \theta)\\
&\ \ -2 (\lambda + \lambda ^3) \cos(3 \theta)+2 \lambda^2 \cos(4 \theta).
\end{align*}
Then,
we may calculate
\begin{align*}
\Phi_{SO(2)}(\lambda) & = \frac{1}{2\pi}\int \frac{1}{\det(\mathds{1}_{2\times 2} - \lambda \mathcal{L}_{\sigma_\theta})} d\theta
\ = \ \frac{1+\lambda ^4}{\left(1-\lambda ^2\right)^2 \left(1-\lambda ^4\right)}\\
&=1+2 \lambda ^2+5 \lambda ^4+8 \lambda ^6+13 \lambda ^8+18 \lambda ^{10}+25 \lambda ^{12} + \cdots
\end{align*}

From the form of the series, 
one can infer from the coefficient of $\lambda^2$ that there are $2$ invariants of order $2$.  These two could be squared separately or multiplied together to make $3$ invariants of order $4$.  We see from the coefficient of $\lambda^4$ that there however are $5$ invariants of this order, so there must be two other invariants of order $4$.  One can combine these $4$ invariants, two of order $2$, and two of order $4$ in eight ways to make an invariant of order $6$, in accord with the Hilbert function.  

However, if these four invariants (two of order $2$ and two of order $4$) were independent, there would be fourteen invariants of order $8$, yet the Hilbert series gives only thirteen, indicating a single `syzygy' at order $8$, meaning a relation between the four invariants.

\medskip
Turning now to the representation of the $(\theta, \epsilon)\in O(2)$ action, note that $\mathcal{L}_{(\theta, \epsilon)} = \mathcal{L}_\sigma$ when $\epsilon = 0$, say when there is no reflection applied.  
Recall the reflection operator $\mathcal{N}$ taking $x_1$ to $x_2$ and $x_2$ to $x_1$ corresponds, in terms of the coefficients of the tensor, to the map $\mathcal{N}(a_0, a_1, a_2, a_3) = (a_3,a_2,a_1,a_0)$, yielding a matrix (also denoted by $\mathcal{N}$), 
\begin{align}
\label{reflection}
\mathcal{N} = \left(\begin{array}{cccc}0&0&0&1\\0&0&1&0\\0&1&0&0\\1&0&0&0\end{array}\right).
\end{align}
Then, when $\epsilon =1$,
$$\mathcal{L}_{(\theta, \epsilon)} = \mathcal{N}\mathcal{L}_\sigma .$$

One may calculate 
$\det(\mathds{1}_{2\times 2} - \lambda \mathcal{L}_{(\theta, 1)}) = (1 - \lambda^2)^2$,
without any $\theta$ dependence.  Hence, 
$$\frac{1}{2\pi}\int \frac{1}{\det(\mathds{1}_{2\times 2} - \lambda \mathcal{L}_{(\theta, 1)})} d\theta = \frac{1}{(1 - \lambda^2)^2}.$$

Consequently
\begin{align}
\label{Molien O(2)}
\Phi_{O(2)}(\lambda) = & \frac{1}{4\pi}\int \frac{1}{\det(\mathds{1}_{2\times 2} - \lambda \mathcal{L}_{(\theta, 0)})} d\theta + \frac{1}{4\pi}\int \frac{1}{\det(\mathds{1}_{2\times 2} - \lambda \mathcal{L}_{(\theta, 1)})} d\theta  \nonumber \\
& = \frac{1}{2} \Phi_{SO(2)}(\lambda) +  \frac{1}{2(1 - \lambda^2)^2} \ = \ 
   \frac{1}{(1-\lambda^2)^2(1-\lambda^4)}\nonumber\\
	&= 1+2 \lambda ^2+4 \lambda ^4+6 \lambda ^6+9 \lambda ^8+12 \lambda ^{10}+16 \lambda^{12}+\cdots
\end{align}
From the generating function for the Hilbert series expansion, we see that there are three invariants, two of degree $2$, and one of degree $4$ that are independent, without syzygies (unlike in the $SO(2)$ setting).

\subsection{ Integrity bases with respect to $SO(2)$ and $O(2)$ }
\label{what are inv n=2}
We now make use of the covariant decomposition discussed earlier to derive an integrity basis with respect to $O(2)$.  At the end of the section, we give another integrity basis used in Remark \ref{rem:6.2} to give an `applied invariant theory' proof of Proposition \ref{prop:2.2}.

Recall the discussion in Section \ref{calc inv section} and, when $n=2$, that 
$f(x;\Gamma) = a_3 x_1^3+3 a_2 x_1^2 x_2+3 a_1 x_1 x_2^2+a_0 x_2^3$ for a general tensor $\Gamma$.  Recall also its $O(2)$ covariant trace vector $u = \nabla (\Delta f)/6= (a_1+a_3, a_0+a_2)$.

Recall also 
\begin{align}
\label{f_1 eq}
f_1 = \frac{3}{4}(u\cdot x)\|x\|^2
= \frac{3}{4}\big((a_1+a_3)x_1 + (a_0+a_2)x_2\big)\|x\|^2,
\end{align}
 corresponding to a tensor $\mathcal{B}$ with the same trace vector $u$ (cf. \eqref{u eqn}).
We also form 
$f_3 =4\big(f - f_1\big)$, which has trace $\frac{1}{6}\nabla \big(\Delta f_3\big)=0$, with corresponding zero trace tensor $\D$ in form
\begin{align}
\label{D eq}
\D^1 = \left(\begin{array}{cc}-3 a_1+a_3& -a_0+3 a_2\\-a_0+3 a_2 &3 a_1-a_3\end{array}\right),\ \  \D^2 = \left(\begin{array}{cc}-a_0+3 a_2 &3 a_1-a_3\\3 a_1-a_3 &a_0-3 a_2\end{array}\right).
\end{align}
The tensor corresponding to $\mathcal{B}$ is found from $\mathcal{B}= \Gamma - \frac{1}{4}\D$.

Recall also the covariant vector $w$ (cf. \eqref{eq:3.5}) given by $w\cdot z = \D(u,u,z)$ for all $z \in \R^n$ in form
\begin{align*}
w &= \big( -3 a_1^3-5 a_3 a_1^2+a_0^2 a_1+9 a_2^2 a_1-a_3^2 a_1+10 a_0 a_2 a_1+a_3^3-3 a_0^2 a_3+5 a_2^2 a_3+2 a_0 a_2 a_3 ,\\ 
   &\ a_0^3-a_2 a_0^2+5 a_1^2 a_0-5 a_2^2 a_0-3 a_3^2 a_0+2 a_1 a_3 a_0-3 a_2^3+a_2 a_3^2+9 a_1^2 a_2+10 a_1 a_2 a_3\big).
   \end{align*}

We now form a number of $SO(2)$ invariants:  The first is $j_2 = \|u\|^2$.  Another is the trace, $h_2 = \sum_k (\D)^{*2}_{k,k}=\sum_{i,j,k=1,2}(\D)^i_{j,k}(\D)^i_{j,k}$.  One more is $\ell_4 = u \cdot w = \D(u, u, u) = \sum_{i,j,k=1,2}(\D)^i_{j,k}u_iu_ju_k$.
Finally, consider the $2\times 2$ matrix $[u^t, w^t]$ and let $m_4 = \det [u^t, w^t]$.

Explicitly,
\begin{align*}
j_2&=(a_0+a_2)^2+(a_1+a_3)^2,\\
h_2&=(a_0-3 a_2)^2+3 (-a_0+3 a_2)^2+3 (3 a_1-a_3)^2+(-3 a_1+a_3)^2,\\
\ell_4&=a_0^4+3 a_1^4-3
a_2^4-8 a_1^3 a_3+24 a_1 a_2^2 a_3+6 a_2^2 a_3^2+a_3^4\\
&\ \ \ -8 a_0 a_2 \left(-3 a_1^2+a_2^2-3 a_1 a_3\right)+6 a_0^2 \left(a_1^2-a_2^2-a_3^2\right)+6
a_1^2 \left(3 a_2^2-a_3^2\right),\\
m_4&=4 (-3 a_0^2 a_1 a_2+a_0^3 a_3+a_2 \left(3 a_1^3+6 a_1^2 a_3-2 a_2^2 a_3-3 a_1 \left(a_2^2-a_3^2\right)\right)\\
&\ \ \ +a_0
\left(2 a_1^3-6 a_1 a_2^2+3 a_1^2 a_3-a_3 \left(3 a_2^2+a_3^2\right)\right)).
\end{align*}

Note that $j_2, h_2$ and $\ell_4$ are invariant with respect to $O(2)$ as they are given in terms of the coefficients of the $O(2)$ covariant characteristic polynomial of $\D^{*2}$.  However, $m_4$ is not invariant with respect to $O(2)$ action,  Indeed, under the reflection $x_1 \to x_2,x_2 \to x_1$ where $(a_0, a_1, a_2, a_3) \to (a_3, a_2, a_1, a_0)$, we have that $m_4 \to - m_4$.

Nevertheless, via use of the identity $(u\cdot w)^2 + \big(\det [u^t, w^t]\big)^2 = \|u\|^2\|w\|^2$, or a Groebner basis computation where the parameters $a_0, a_1, a_2, a_3$ are eliminated to see relations among $j_2, h_2, \ell_4, m_4$, one finds 
\begin{align*}
-h_2 j_2^3+4 \ell_4^2+4 m_4^2=0.
\end{align*}
Hence, one of the invariants of degree $4$, when squared, say $m_4^2$ is expressible in terms of the $SO(2)$ invariants $h_2, j_2$ and $\ell_4$.  This is the syzygy mentioned earlier in the context of $SO(2)$ invariants.  Hence, consistent with Molien's formula, $h_2, j_2, \ell_4, m_4$ are an integrity basis with respect to $SO(2)$.  

With respect to $O(2)$ however, we conclude that $h_2, j_2$ and $\ell_4$, as they are independent, form an integrity basis.  

Here, although $m_4^2$ is polynomially dependent on $h_2, j_2, \ell_2$ in terms of the syzygy, $m_4$ is not expressible as a polynomial in terms of $h_2, j_2, \ell_4$.   Since $O(2)$ orbits may consist of two disconnected $SO(2)$ orbits, the role of the $SO(2)$ invariant $m_4$ is to distinguish which $SO(2)$ orbit $\Gamma\in \mathcal{T}_2$ would be on.

Finally, we comment that, in terms of 
\begin{align}
\label{n=2 Gamma *2}
\Gamma^{*2}= \left(\begin{array}{cc} a_1^2 + 2a_2^2 + a_3^2& a_0a_1 + a_2(2a_1+a_3)\\
a_0a_1 + a_2(2a_1+a_3)& a_0^2 + 2a_1^2 + a_2^2\end{array}\right),
\end{align}
the following polynomials also form an integrity basis with respect to $O(2)$:
\begin{align*}
\big\|\big({\rm Trace}\; \Gamma\big)\big\|^2&=j_2, \ \ 
{\rm Trace}\; \Gamma^{*2}=\frac{1}{16}\big(h_2+12j_2\big), {\rm \ \ and \ \ }\\
\det \Gamma^{*2}&= \frac{1}{1024}\big(h_2^2 + 8h_2j_2 + 80j_2^2 - 128\ell_4\big).
\end{align*}

\section{Membership problem for $O(n)$ invariant subsets of $\mathcal{T}_n$}
\label{membership section}
  We discuss a scheme to characterize $O(n)$ orbits of invariant subsets $\mathcal{S}$ of $\mathcal{T}_n$.
Recall that a subset $\mathcal{S}\subset \mathcal{T}_n$ is invariant if $\sigma\circ\Gamma\in \mathcal{S}$ for all $\sigma\in O(n)$ and $\Gamma\in \mathcal{S}$.

  Consider a subset of tensors $\mathcal{R}\subset \mathcal{T}_n$.   A useful way to obtain an $O(n)$ invariant subset $\mathcal{S}$, containing $\mathcal{R}$, is to form $\mathcal{S}=\{\sigma\circ R: R\in \mathcal{R}, \sigma\in O(n)\}=\cup_{R\in \mathcal{R}} \{O(n) \ {\rm orbit \ of \ }R\}$.  
 It will be useful to define the `stabilizer' subgroup $G_\RR$ in $O(n)$ which stabilizes $\mathcal{R}$:  
 That is, if $g\in G_\RR$, then $g\circ R \in \mathcal{R}$ for all $R\in \mathcal{R}$.

 The scheme requires the following ingredients:
  (1) a generating set $\mathcal{I}$ for the $O(n)$ invariants of $\mathcal{T}_n$, and
 (2) a generating set $\mathcal{J}$ for the $G_\RR$ invariants on $\mathcal{R}$.
 
Note that every $O(n)$ invariant in $\mathcal{I}$ is invariant for the action of $G_\RR$ on $\mathcal{R}$ because $G_\RR$ is a subgroup of $O(n)$.
This implies that every invariant in $\mathcal{I}$ can be written as polynomials in terms of the quantities in $\mathcal{J}$. If we can isolate the members of $\mathcal{J}$ in these relations, then we naturally get extensions of the $G_\RR$ invariants in $\mathcal{J}$ to all of $\mathcal{S}$ since they are now given in terms of the $O(n)$ invariants in $\mathcal{I}$ whose domain of definition necessarily contains $\mathcal{S}$. More generally, we will require that
 there is a set $\widetilde{\mathcal{J}}$ of $O(n)$ invariants on an invariant set $\widetilde{\mathcal{T}}$ containing $\mathcal{S}$ whose restriction to $\mathcal{R}$ gives the $G_\RR$ invariants $\mathcal{J}$.

 We comment that 
  $\widetilde{\mathcal{J}}$, although composed of $O(n)$ invariants, may not all be in polynomial form.
  In particular, $\widetilde{\mathcal{T}}\supset \mathcal{S}$ may be a strict subset of $\mathcal{T}_n$, say when the members of $\widetilde{\mathcal{J}}$ are rational invariants or other functions with singularities.

We now state a characterization of when $\Gamma\in \mathcal{T}_n$ belongs to the $O(n)$ invariant set $\mathcal{S}$ in terms of a `lifted' set of relations and a `solvability' condition with respect to the `small' set $\mathcal{R}$.  This characterization also will identify a member $R\in \mathcal{R}$ so that $\Gamma \sim R$.

\begin{thm}
\label{thm:meta}
Let $\mathcal{I}=\{I_i\}$ be a generating set of invariants for $O(n)$ orbits in $\mathcal{T}_n$.  Let $\mathcal{S} \subset \mathcal{T}_n$ be an invariant subset formed as the union of $O(n)$ orbits of $\mathcal{R}$, where $R\in \mathcal{R}$ is specified in terms of $r$ parameters $v_1,\ldots, v_{r}$. 
Let also $\mathcal{J}=\{q_i\}$ be a generating set of invariants with respect to $G_\RR$ action on $\mathcal{R}$, and suppose that an extension set $\widetilde {\mathcal{J}}=\{\widetilde q_i\}$ of $O(n)$ invariants exists in general, not on $\mathcal{T}_n$, but on $\widetilde{\mathcal{T}}$ such that $\mathcal{S}\subset\widetilde{\mathcal{T}}\subset \mathcal{T}_n$.

We will impose a `solvability' condition:  Suppose, for $\Gamma\in \widetilde{\mathcal{T}}$, there is an $R\in \mathcal{R}$ such that
\begin{align}
\label{R relations}
\widetilde q_i(\Gamma) = \widetilde q_i(R), 
{\rm \ for \ } i=1,\ldots, |\mathcal{J}|.
\end{align}

Since $\mathcal{I}$ restricted to $\mathcal{R}$ are $G_\RR$ invariants, and $\mathcal{J}$ is a generating set of such invariants, each member of $\mathcal{I}|_{\mathcal{R}}=\{I_i\}_{\mathcal{R}}$ can be written as a polynomial relation of $\mathcal{J}$.  That is, for polynomials $\{F_i\}$, 
we have on $\mathcal{R}$ that $I_i = F_i(\{q_j\})$ for $i=1,\ldots, |\mathcal{I}|$.  
These relations when lifted to $\widetilde{\mathcal{T}}$ 
prescribe another condition:
\begin{align}
\label{extended relations}
I_i = F_i(\{\widetilde q_j\}) \ {\rm for \ } i = 1,\ldots,  |\mathcal{I}|.
\end{align}

Finally, let $\Gamma\in \mathcal{T}_n$.  Then, we have that $\Gamma\in \mathcal{S}$ exactly when (1) $\Gamma\in \widetilde{\mathcal{T}}$ and the relations on $\Gamma$ in condition \eqref{extended relations} hold, and (2) the solvability condition \eqref{R relations} holds for $\Gamma$.
\end{thm}

\begin{proof} First, given a tensor $\Gamma \in \mathcal{S}$, as $\Gamma$ is on an $O(n)$ orbit of a tensor $\hat R\in\mathcal{R}$, the values of the $O(n)$ invariants 
$\widetilde q_i(\Gamma) = \widetilde q_i(\hat R)$.
Hence, solvability \eqref{R relations} clearly holds with respect to this $\hat R$ and is therefore necessary.  Also, by $O(n)$ invariance $I_i(\Gamma)=I_i(\hat R)$ and by construction $I_i(\hat R) = F_i(\{q_j(\hat R)\})$ which equals $F_i(\{\widetilde q_j(\hat R)\})$.
Then, for this $\Gamma\in \mathcal{S}\subset \widetilde{\mathcal{T}}$, the
relations \eqref{extended relations} hold, and are therefore necessary.

On the other hand, suppose $\Gamma\in \widetilde{\mathcal{T}}$ satisfies \eqref{extended relations}, and that the solvability condition \eqref{R relations} holds.  Then, there is a tensor $\hat R \in \mathcal{R}$ such that $\widetilde q_i(\Gamma) = \widetilde q_i(\hat R)$
for $i=1,\ldots, |\mathcal{J}|$. 
By the relations \eqref{extended relations}, the values $I_i(\Gamma)=F_i(\{\widetilde q_j(\Gamma)\})$. 
By solvability \eqref{R relations}, since $\widetilde q_i(\Gamma) = \widetilde q_i(\hat R)$, we have that $I_i(\Gamma)=I_i(\hat R)$ for each member of $\mathcal{I}$.  
 Since $\mathcal{I}$ is a generating set with respect to $O(n)$ action on $\mathcal{T}_n$, it separates orbits by Proposition \ref{prop:separating}.  Therefore, the tensor $\Gamma$ must be on the same $O(n)$ orbit as $\hat R\in \mathcal{R}$, 
 and therefore $\Gamma$ belongs to $\mathcal{S}$.
 \end{proof}

\begin{rem} 
\label{rem:semi-algebraic}
\rm 
The solvability requirements implied by  \eqref{R relations} have to be obtained from additional considerations depending on the context in which this theorem is applied. These requirements can include equalities, for instance any syzygies among the generating set $\mathcal{J}$ of $G_\RR$ invariants, as well as inequalities, reflecting the fact that we want a real tensor $R$ as the solution to \eqref{R relations}.

In general, the set $\mathcal{S}$, given as a union of orbits, is not a variety, that is it cannot be characterized purely by a set of polynomial equations. 
The approach in Theorem \ref{thm:meta} has been to obtain an implicit description of $\mathcal{S}$ as a `semi-algebraic' set, that is a subset of a real vector space defined by a collection of polynomial equations and polynomial inequalities (cf. \cite{RAGBook}).
\end{rem}

\subsection{Solving for a transformation $\sigma$ to reach the canonical form in $\mathcal{R}$}
Although we have derived necessary and sufficient conditions for a tensor $\Gamma \in \mathcal{T}_n$ to be in $\mathcal{S} = \cup_{R\in \mathcal{R}}\big\{O(n){\rm \ orbit \ of \ }R\big\}$, the approach taken does not need to identify the transformation $\sigma\in O(n)$ so that $\sigma \circ \Gamma =R\in \mathcal{R}$.  
When $n=2$, by the direct and ODE methods (cf. Sections \ref{prop:2.2 section}, \ref{ODE section}), solving for $\sigma$ was part of the solution.

Here, we discuss a way to obtain $\sigma$ `generically' in the more abstract context of Theorem \ref{thm:meta}.  
We first state a standard `linear algebra' lemma, accomplished by orthogonal diagonalization.

\begin{lem}
Let $Q$ be an $n \times n$ symmetric, positive semi-definite matrix with distinct eigenvalues, and let $Q'$ be another $n \times n$ symmetric matrix with the same set of eigenvalues.  Then, there exist exactly $2^n$ distinct orthogonal matrices $\rho \in O(n)$ such that $\rho Q \rho^t = Q'$. 
\label{lem:ortho}
\end{lem}

\begin{proof}
Let $0 \leq \lambda_1 < \lambda_2 < \cdots < \lambda_n$ be the eigenvalues of $Q$ and $Q'$ and let $u_i, v_i$ be corresponding orthonormal eigenvectors satisfying $Q u_i= \lambda_i u_i, Q' v_i = \lambda_i v_i$. If $\rho Q\rho^t = Q'$, it follows that $Q' (\rho u_i) = \rho (Q u_i) = \lambda_i (\rho u_i)$ implying that $\rho u_i = \epsilon_i v_i$ for $i = 1,2,\ldots,n$. Since $\|u_i\| = \|v_i\| = 1$, we have that $\epsilon_i^2 = 1$. These relations give $\rho[u_1,\ldots, u_n] = [\epsilon_1v_1, \ldots, \epsilon_nv_n]$ or in other words $\rho = \sum_i \epsilon_i v_i u_i^t$. Conversely, for each of the $2^n$ choices of $\mathbf{\epsilon} = (\epsilon_1,\epsilon_2,\ldots,\epsilon_n) \in \{-1,1\}^n$, setting $\rho_{\mathbf \epsilon} = \sum_i \epsilon_i v_i u_i^t$ gives distinct orthogonal matrices $\rho_{\mathbf{\epsilon}}$ such that $\rho_{\mathbf \epsilon} Q \rho_{\mathbf \epsilon}^t = Q'$. 
\end{proof}

Let $\Ga_1 \sim \Ga_2$ and  let $Q$ be a symmetric matrix valued  covariant (cf.  Definition~\ref{vector cov defn}). An immediate application of the preceding lemma yields the following procedure to `generically' find a transformation $\sigma \in O(n)$ such that $\sigma \circ \Ga_1 = \Ga_2$.

\begin{thm} \label{rem:rotation} Let $\Ga_1,\Ga_2\in \mathcal{T}_n$ be two tensors on the same $O(n)$-orbit.  Suppose that $\sigma \circ \Ga_1 = \Ga_2$ for $\sigma\in O(n)$.
Let $Q$ be a symmetric matrix valued covariant. 
 If the `genericity' assumption that 
 the eigenvalues of $Q(\Ga_2)$ are distinct holds,  then $\sigma$ is one of the finite list of $2^n$ orthogonal maps $\rho$ where $\rho Q(\Ga_1)\rho^t =Q(\Ga_2)$.
 \end{thm}
 
 \begin{proof}
 Since $Q$ is covariant, we have that $x^t Q(\sigma\circ \Gamma_1)x = (\sigma^t x)^t Q(\Gamma_1)(\sigma^t x)$ (cf. Definition \ref{vector cov defn}).  Hence, unraveling the form, $Q(\Ga_2)=Q(\sigma\circ \Gamma_1) = \sigma Q(\Ga_1)\sigma^t$. 
 Then, by Lemma~\ref{lem:ortho}, it follows that a transformation $\sigma$ yielding $\sigma \circ \Ga_1 = \Ga_2$ can be found by exhaustively searching through a finite list of orthogonal transformations.  
\end{proof}

To apply the method outlined in the proof of Theorem \ref{rem:rotation}, we need to have at hand a symmetric matrix-valued covariant. Examples of such suitable covariants include ${\Ga^*}^2, u \cdot \Ga, {\mathcal{D}^*}^2$ etc.  In the following, we generically identify $\sigma\in O(n)$ when $\sigma\circ \Gamma$ is in either fully decoupled or partially decoupled form.

\begin{cor}
\label{cor:finding}
Consider the matrix covariant $\Gamma^{*2}$.  

(1) Suppose $\Gamma\sim \Gamma^{(\beta_1, \ldots, \beta_n)}$ is fully decoupleable.  When $\beta_1^2, \ldots, \beta_n^2$ are distinct, in other words the eigenvalues of $(\Gamma^{(\beta_1, \ldots, \beta_n)})^{*2}$ are distinct, by Theorem \ref{rem:rotation}, any map $\sigma\in O(n)$ such that $\sigma\circ \Ga = \Ga^{(\beta_1,\ldots, \beta_n)}$ is one of the $2^n$ transformations $\rho\in O(n)$ where $\rho^t \Gamma^{*2}\rho = (\Gamma^{(\beta_1,\ldots, \beta_n)})^{*2}$.

(2) Suppose $\Gamma\sim R\in \mathcal{R}$, where $R$ is in partially decoupled form \eqref{partial reduced}.   When $R^{*2}$ has distinct eigenvalues, by Theorem \ref{rem:rotation} again, any map $\sigma\in O(n)$ such that $\sigma\circ \Gamma = R$ is one of the $2^n$ maps $\rho\in O(n)$ where $\rho^t \Gamma^{*2}\rho = R^{*2}$.
\end{cor}

We comment that the condition that the eigenvalues are distinct is crucial to the preceding discussion and not merely a technical condition; see Example \ref{ex:not distinct}.  We also remark that to find $\sigma$ in `non-generic' settings of say repeated eigenvalues of a covariant matrix $Q$ seems to require more particular calculations in terms of the forms of $Q$ and $R$, which we leave for future consideration.

\begin{exa}
\label{ex:not distinct}\rm
 Consider the $2 \times 2 \times 2$ tensor $\Ga$ given by $f(x;\Ga) = x_1^3 - x_2^3$. Here, $\Ga$ is already in fully decoupled form and the matrix covariant $Q = {\Ga^*}^2 =\mathds{1}_{2\times 2}$ has repeated eigenvalues, $\beta_1^2=\beta_2^2=1$. 
Clearly, the orthogonal matrix $\displaystyle{\sigma = \frac{1}{\sqrt{2}} \begin{pmatrix} 1 & -1 \\ 1 & 1 \end{pmatrix}}$ is such that $Q(\sigma\circ \Gamma)=\sigma Q(\Ga) \sigma^t=\mathds{1}_{2\times 2}$. However, an explicit calculation yields $f(x; \sigma \circ \Ga) = f(\sigma^{-1} x; \Ga) = \frac{1}{\sqrt{2}} (3 x_1^2 x_2 + x_2^3)$, corresponding to $\sigma \circ \Ga$ not in a fully decoupled form.  Hence, we have a `violation' of Theorem \ref{rem:rotation} in that $Q(\eta\circ \Gamma)$ is diagonal for all $\eta\in O(2)$, but $\sigma\circ \Gamma$ is not in fully decoupled reduced form.  The `violation' arises because there is ambiguity in determining eigenvectors with respect to the repeated eigenvalue.
\end{exa}

  \section{Full decoupleability relations via symmetric polynomials for $n\geq 2$}
  \label{full decoup sym proof}
  
  We will make use of Theorem \ref{thm:meta} to characterize the subset of fully decoupleable tensors $\mathcal{S}$ in the space of trilinear $\mathcal{T}_n$ for $n\geq 2$.  Consider the reduced `diagonal' representation $R$ of a fully decoupleable tensor, where $R^i_{j,k}=0$ unless $i=j=k$ for $i=1,\ldots, n$.
  We will specify that $\mathcal{R}$ consists of these reduced representatives.
  Then, $\mathcal{S}= \cup_{R\in \mathcal{R}} \{O(n) \ {\rm orbit \ of \ }R\}$.

  On $\mathcal{R}$, let $G_\RR$ be the $O(n)$ action with respect to the $n$ diagonal entries.  In other words, $G_\RR = S_n\times (\Z_2)^n$ corresponding to permutation of the diagonal entries and also changing their signs.
  
 Consider the characteristic polynomial $\mathfrak{p}_\Gamma$ of $\Gamma^{*2}$ for $\Gamma\in \mathcal{T}_n$.  The $n$ coefficients of $\mathfrak{p}_\Gamma$ are symmetric functions of the eigenvalues of $\Gamma^{*2}$ which are real and nonnegative as $\Gamma^{*2}$ is symmetric and nonnegative definite.  Since the quadratic form of $\Gamma^{*2}$ is $O(n)$ covariant, these coefficients $\widetilde{\mathcal{J}}= \{\widetilde q_i\}_{i=1,\ldots, n}$ are invariants on $\mathcal{T}_n$. Here, $\widetilde{\mathcal{T}}=\mathcal{T}_n$. Restricted to $R\in \mathcal{R}$, the eigenvalues of $R^{*2}$ are the squares of the $n$ diagonal entries of $R$, and $\mathcal{J}:= \widetilde{\mathcal{J}}|_{\mathcal{R}}$ consist of the elementary symmetric functions $\{q_i\}_{i=1,\ldots, n}$ of these squares, which are invariant to the $G_\RR$ action.  
 
 Now, any $G_\RR$ invariant on $\mathcal{R}$, as it is invariant to permutation of diagonal elements and sign changes, must be a symmetric function of the squares of the diagonal elements.  Since $\mathcal{J}$ consist of the elementary symmetric functions, these generate all other $G_\RR$ polynomial invariants (cf. Chapter 1 in \cite{Sturmfels}).

  \begin{thm}
  \label{thm:full n}
  Given a generating set of invariants $\mathcal{I}=\{I_i: 1=1,\ldots, d_n\}$ of $O(n)$ orbits in $\mathcal{T}_n$, we may find $d_n$ polynomial conditions \eqref{full eq},
  which are necessary and sufficient to identify $O(n)$ orbits of fully decoupleable tensors in $\mathcal{T}_n$.
  \end{thm}
  
  \begin{proof}
  We will apply Theorem \ref{thm:meta} in the current context.
 
  We first show that the solvability condition \eqref{R relations} always holds.  With respect to $\Gamma\in \mathcal{T}_n$, let $z_1,\ldots z_n$ be the real, nonnegative eigenvalues of $\Gamma^{*2}$. The tensor in fully decoupleable reduced form $R$ with respect to say diagonal elements $R^1_{1,1}=\sqrt{z_1},\ldots, R^n_{n,n}=\sqrt{z_n}$ is such that $R^{*2}$ has the same eigenvalues as $\Gamma^{*2}$. Therefore,
  $\widetilde q_i(\Gamma) = \widetilde q_i(R) = q_i(R)$ for $i=1, \ldots, n$, as desired.
  
  We now observe that $I_i$, when restricted to $\mathcal{R}$, is invariant to $G_\RR$ action.   Since $\mathcal{J}$ generates all polynomial $G_\RR$ invariants on $\mathcal{R}$, we may write $I_i=F_i(\mathcal{J})$ as a polynomial function of the elements of $\mathcal{J}$.
  
Hence, by Theorem \ref{thm:meta}, we conclude these relations, when lifted to $\mathcal{T}_n$, that is 
\begin{align}\label{full eq}
I_i = F_i(\widetilde{\mathcal{J}}) \ \ \ {\rm for\  }i=1,\ldots, d_n,
\end{align}
 are necessary and sufficient to characterize when $\Gamma\in \mathcal{T}_n$ belongs to an $O(n)$ orbit of $\mathcal{R}$, that is in $\mathcal{S}$.
\end{proof}

\begin{rem}
\label{rem:6.2}
\rm
When $n=2$, by the work in Section \ref{what are inv n=2}, we know that an $O(2)$ integrity basis consists of $I_1 = \|\big({\rm Trace} \; \Gamma\big)\|^2$, $I_2 = {\rm Trace}\; \Gamma^{*2}$ and $I_3=\det \Gamma^{*2}$.  Also, the coefficients of the characteristic polynomial of $\Gamma^{*2}$ are $\widetilde q_1 = I_2$ and $\widetilde q_2 = I_3$.  
By evaluating on reduced fully decoupleable forms, we observe $I_1=F_1(\{\widetilde q_j\})=\widetilde q_1$, $I_2=F_2(\{\widetilde q_j\})=\widetilde q_1$ and $I_3=F_3(\widetilde q_j\}) = \widetilde q_2$.  Then, according to Theorem \ref{thm:full n}, the necessary and sufficient relations reduce to
 $I_1 = I_1$, $I_2=I_1$ and $I_3=I_3$, or that the necessary and sufficient relation is $I_1=I_2$.  In terms of parameters, noting \eqref{n=2 Gamma *2}, we have
 \begin{align*}
 (a_1+a_3)^2 + (a_0+a_2)^2 = a_0^2 + 3a_1^2 + 3a_2^2 + a_3^2,
 \end{align*}
which reduces to the condition known already in Proposition \ref{prop:2.2}.
\end{rem}

\section{Explicit characterization of $O(n)$ orbits of fully and partially decoupled tensors when $n=3$}
\label{explicit n=3}

After specifying the Olive-Auffray integrity basis \eqref{OA basis} when $n=3$ in Section \ref{OA section}, we use it with respect to scheme of Theorem \ref{thm:meta} in the next two sections to develop more `explicit' criteria for membership, using the list of relations in the last section.

  \subsection{Olive and Auffray's integrity basis} 
  \label{OA section}

In 2014, Olive and Auffray \cite{OA} derived an integrity basis with respect to $O(3)$ action on real $3\times 3\times 3$ trilinear tensors $\Gamma$, in the context of elasticity problems, via connection of the $O(3)$ action on real trilinear tensors to that of $SL(2, \mathbb{C})$ action on bilinear forms over the complex vector space $\mathbb S_6\oplus \mathbb S_2$, where $\mathbb S_{2k}$ denotes the (complex) vector space of homogeneous polynomials in two complex variables of total degree $2k$.

Recall the discussion of decomposing $\Gamma = \frac{1}{n+2}\D+\mathcal{B}$ into the sum of a trace-free tensor $\D$ and a tensor $\mathcal{B}$ determined entirely by the  trace vector $u$ (near \eqref{u eqn}), with components by $u_i = \Gamma^j_{i,j} := \sum_{j=1}^3 \Gamma^j_{i,j}$.  
  In the context of $n=3$, 
Olive and Auffray construct two other vectors $v,w$ that are determined by $\Gamma$ and are covariant with respect to the $O(n)$ action, namely
\begin{align*}
v_m= \D^i_{j,k}\D^i_{j,\ell}\D^k_{\ell, m}, \qquad w_m = \D^i_{j,m}u_iu_j,
\end{align*}
using the Einstein summation convention.
Generically, that is when $\det [u, v, w]\neq 0$, the set $(u,v,w)$ is inearly independent, so that it specifies a canonical (covariant) frame for $\mathbb{R}^3$, determined entirely by $\Gamma$.

We now recall Theorem 2.2 in \cite{OA}, stated in our notation for general completely symmetric tensors.   
The integrity basis with respect to $O(3)$ identified in \cite{OA}, consisting of $d_3=13$ members, is the following, using the label $H$ for $I$ (since $I$ is reserved in Mathematica for an accompanying nb file) and again the Einstein summation convention:

\begin{equation}
\label{OA basis}
\begin{tabular}{lll}
$H_2 = \D^i_{j,k}\D^i_{j,k}$ & $J_2 = u_i^2$ & $H_4 = \D^i_{j,k}\D^i_{j,\ell}\D^p_{q, k}\D^p_{q, \ell}$ \\
$J_4 = \D^i_{j,k}u_k \D^\ell_{j,\ell}u_\ell$ & $K_4 = \D^i_{j,k}\D^i_{j,\ell}\D^k_{\ell,p}u_p$ & $L_4 = \D^i_{j,k}u_ku_j u_i$\\
$H_6 = v_i^2$ & $J_6 = \D^i_{j,k}\D^i_{j,\ell}u_k \D^\ell_{p,q}u_pu_q$ & $K_6 = v_kw_k$\\
$L_6 =\D^i_{j,k}\D^i_{j,\ell}u_k v_\ell$ & $M_6 = \D^i_{j,k}\D^p_{q,k} u_i u_j u_p u_q$
&$H_8 = \D^i_{j,k}\D^i_{j,\ell}u_k\D^p_{q, \ell}\D^p_{q, r}v_r$\\
$H_{10}= \D^i_{j,k}v_iv_jv_k$. &&
\end{tabular}
\end{equation}

We remark that the invariants are homogeneous polynomials in the entries of $\Ga$ with degrees corresponding to the subscripts.
Also, we comment that the expressions in \eqref{OA basis} are proportional but not exactly those in \cite{OA}.   In \cite{OA},
the formulas are those in \eqref{OA basis} with $\D' = \frac{1}{n+2}\D$ in place of $\D$.  We have used $\D$ here to avoid denominators.

\subsection{Fully decoupled tensors when $n=3$}

We will apply Theorem \ref{thm:meta} with respect to the explicit integrity basis \eqref{OA basis} when $n=3$ to identify several necessary and sufficient relations.
\begin{thm}
\label{thm:full explicit}
  When $n=3$, we have $13$ necessary and sufficient relations for $\Gamma\in \mathcal{T}_3$ to be fully decoupleable.  These are the following $13$ relations in terms of $\{\widetilde q_i=\widetilde q_i(\Ga)\}_{i=1, 2, 3}$ given in \eqref{wide tilde q}.
  \begin{align}
\label{full relations thm}
H_2(\Gamma)&=10 \widetilde q_1\\
H_4(\Ga)&=2 \left(22 \widetilde q_1^2-15 \widetilde q_2\right)\nonumber\\
J_2(\Ga)&=\widetilde q_1\nonumber\\L_4&=2 \left(\widetilde q_1^2-5 \widetilde q_2\right)\nonumber\\
H_6(\Ga)&=4 \left(16 \widetilde q_1^3-55 \widetilde q_1 \widetilde q_2+75
\widetilde q_3\right)\nonumber\\
H_{10}(\Ga)&=8 \left(128 \widetilde q_1^5-700 \widetilde q_1^3 \widetilde q_2+725 \widetilde q_1 \widetilde q_2^2+950 \widetilde q_1^2 \widetilde q_3-875 \widetilde q_2 \widetilde q_3\right)\nonumber\\
J_4(\Ga)&=2 \left(3 \widetilde q_1^2-5 \widetilde q_2\right)\nonumber\\
K_4(\Ga)&=4 \left(2
\widetilde q_1^2-5 \widetilde q_2\right)\nonumber\\
J_6(\Ga)&=12 \widetilde q_1^3-55 \widetilde q_1 \widetilde q_2+75 \widetilde q_3\nonumber\\
K_6(\Ga)&=2 \left(8 \widetilde q_1^3-35 \widetilde q_1 \widetilde q_2+75 \widetilde q_3\right)\nonumber\\
L_6(\Ga)&=6 \left(8 \widetilde q_1^3-25 \widetilde q_1 \widetilde q_2+25 \widetilde q_3\right)\nonumber\\
M_6(\Ga)&=4
\widetilde q_1^3-15 \widetilde q_1 \widetilde q_2+75 \widetilde q_3\nonumber\\
H_8(\Ga)&=4 \left(72 \widetilde q_1^4-270 \widetilde q_1^2 \widetilde q_2+75 \widetilde q_2^2+325 \widetilde q_1 \widetilde q_3\right).\nonumber
\end{align}
  \end{thm}

\begin{proof} 
Form $\widetilde {\mathcal{J}}=\{\widetilde q_i\}$, composed of the coefficients of the characteristic polynomial of $\Gamma^{*2}$.  Note that $\Gamma^{*2}$ is a symmetric, nonnegative definite matrix.  Write
\begin{align}
\label{wide tilde q}
\widetilde q_3(\Ga) &= \det \Gamma^{*2}=\tilde \beta_1^2\tilde \beta_2^2 \tilde \beta_3^2\\
\widetilde q_2(\Ga) &= \det A_{1,1} + \det A_{2,2} + \det A_{3,3} = \tilde \beta_1^2\tilde \beta_2^2 + \tilde\beta_1^2\tilde\beta_3^2 + \tilde\beta_2^2\tilde\beta_3^2\nonumber\\
  \widetilde q_1(\Ga) &= {\rm Trace}\;  \Gamma^{*2} = \tilde \beta_1^2 + \tilde \beta_2^2 + \tilde \beta_3^2, \nonumber
  \end{align}
   where $A_{i,i}$ are principal $2\times 2$ submatrices of $\Gamma^{*2}$ formed by omitting row $i$ and column $i$,  and $\tilde \beta^2_1, \tilde \beta^2_2, \tilde \beta^2_3$ are the real, nonnegative eigenvalues of $\Gamma^{*2}$.   These are $O(3)$ invariants on all of $\mathcal{T}_3$.
 
Let $\mathcal{R}$ be the set of tensors in fully decoupleable reduced form \eqref{full canonical} with diagonal entries.
On $\mathcal{R}$, the set $\widetilde{\mathcal{J}}$ reduces to $\mathcal{J}=\{q_i\}$, a generating set of invariants with respect to the symmetric group $S_3$ (Theorem 1.1.1 in \cite{Sturmfels}), playing the role of $G_\RR$:  For $R\in \mathcal{R}$ with diagonal entries given by $\{\beta_i\}$,
$$
\q(R) = \beta_1^2 \beta_2^2\beta_3^2, \ \ q_2(R) = \beta_1^2\beta_2^2 + \beta_1^2\beta_3^2 + \beta_2^2\beta_3^2, \ \ \r (R)= \beta_1^2 + \beta_2^2 + \beta_3^2.$$

Hence, solvability $\widetilde q_i(\Gamma) = q_i(R)$ for $i=1,2,3$ (cf. \eqref{R relations}) always holds by taking $R\in \mathcal{R}$ with parameters $\big\{\beta_i=\sqrt{\tilde \beta_i^2}\big\}$, in terms of the eigenvalues of $\Gamma^{*2}$.

Denote the integrity basis in \eqref{OA basis} by $\mathcal{I}$. Via a Groebner basis computation, 
the integrity basis values on a fully decoupled tensor in reduced form $R$ are found in terms of $\{q_i = q_i(R)\}$:
\begin{align}
\label{full relations alternate}
H_2(R)&=10 q_1\\
H_4(R)&=2 \left(22 q_1^2-15 q_2\right)\nonumber\\
J_2(R)&=q_1\nonumber\\L_4&=2 \left(q_1^2-5 q_2\right)\nonumber\\
H_6(R)&=4 \left(16 q_1^3-55 q_1 q_2+75
q_3\right)\nonumber\\
H_{10}(R)&=8 \left(128 q_1^5-700 q_1^3 q_2+725 q_1 q_2^2+950 q_1^2 q_3-875 q_2 q_3\right)\nonumber\\
J_4(R)&=2 \left(3 q_1^2-5 q_2\right)\nonumber\\
K_4(R)&=4 \left(2
q_1^2-5 q_2\right)\nonumber\\
J_6(R)&=12 q_1^3-55 q_1 q_2+75 q_3\nonumber\\
K_6(R)&=2 \left(8 q_1^3-35 q_1 q_2+75 q_3\right)\nonumber\\
L_6(R)&=6 \left(8 q_1^3-25 q_1 q_2+25 q_3\right)\nonumber\\
M_6(R)&=4
q_1^3-15 q_1 q_2+75 q_3\nonumber\\
H_8(R)&=4 \left(72 q_1^4-270 q_1^2 q_2+75 q_2^2+325 q_1 q_3\right).\nonumber
\end{align}
These lifted to $\Gamma\in \mathcal{T}_3$ as in \eqref{full relations thm}, replacing $q_i=q_(R)$ with $\widetilde q_i=\widetilde q_i(\Ga)$, give necessary and sufficient relations 
for a tensor $\Gamma\in \mathcal{T}_3$ to be fully decoupleable, as the solvability condition \eqref{R relations} always holds.  
\end{proof}

\subsection{Partially but not fully decoupleable tensors when $n=3$}
\label{explicit partial n=3}

Let $\mathcal{R}$ be the set of partially but not fully decoupleable tensors in reduced form (cf. \eqref{partial reduced}) in $\mathcal{T}_3$.  Such forms are given in terms of parameters $a_3, a_2, a_1, a_0$ and $\beta_3$.  
Recall that one of the three matrices, say $\Gamma^3$, is fixed to be a diagonal matrix with $\beta_3 = \Gamma^3_{3,3}$, $\Gamma^3_{2,2}=\Gamma^3_{1,1}=0$; of course, by an $O(3)$ transformation, the diagonal matrix could have been either $\Gamma^2$ or $\Gamma^1$. 
What is left in the description of $\Gamma$ is a $2 \times 2 \times 2$ subtensor, say $G^1, G^2$, which we may write in terms of $a_i$ for $i=0,1,2, 3$, as in the reduced form \eqref{partial reduced}.

A tensor in reduced form is partially but not fully decoupleable exactly when $a_2(a_2-a_0)\neq a_1(a_3-a_1)$ (cf. Lemma \ref{rem:3.2}). 
Let $\mathcal{S}=\cup_{R\in \mathcal{R}}\{O(3) \ {\rm orbit \ of \ } R\}$ be the collection of partially but not fully decoupled tensors $\mathcal{T}_{3, PD}\setminus \mathcal{T}_{3, FD} \subset \mathcal{T}_3$.

Define the stabilizer subgroup $G_\RR=O(2)\times \Z_2$ of $O(3)$, identifying members of $O(3)$ whose action on a reduced partially but not fully decoupleable tensor would remain in $\mathcal{R}$.  The $O(2)$ part refers to transformations of the $2\times 2\times 2$ subtensor given in terms of $a_3, a_2, a_1, a_0$ and $\Z_2$ refers to changing sign of the $\beta_3$ parameter.

We would like to identify a generating set $\mathcal{J}$ for the $G_\RR$ invariants on $\mathcal{R}$.  
Invoking Molien's formula, noting that the Molien function with respect to a direct product is a product of the Molien functions of the factors (cf. Lemma 2.2.3 in \cite{Sturmfels}), and \eqref{Molien O(2)}, we have 
$$\Phi_{O(2)\oplus \Z_2}(\lambda) = \frac{1}{1-\lambda^2}\Phi_{O(2)}(\lambda)=\frac{1}{(1-\lambda^2)^3(1-\lambda^4)}.$$
Hence, one should look for a generating set, or integrity basis, of three invariants of degree $2$ and one of degree $4$.

We now give a useful `canonical' form for a partially decoupled tensor $\Gamma\in \mathcal{T}_3$.  Recall the `reduced' form in \eqref{partial reduced}.  Let $G=\{G^1, G^2\}$ be the covariant $2\times 2\times 2$ subtensor, seen in the upper left of $\Gamma^1, \Gamma^2$, and let $\beta_3$ be the sole diagonal element in $\Gamma^3$.  

Let $u_G= (a_3 + a_1, a_2+a_0)$ be the trace vector of $G$, and $u^\perp_G= (a_2 + a_0, -a_3-a_1)$ be a perpendicular vector.  If $\|u_G\| \neq 0$, let $z=(z_1, z_2)$ where $z_1 = \frac{1}{\|u_G\|}(u_G\cdot x)$ and $z_2= \frac{1}{\|u_G\|}(u_G^\perp \cdot x)$.  Consider the rotation $\zeta =\zeta^{-1}=\|u_G\|^{-1}(u_G^t,(u_G^\perp)^t)\in O(2)$ which takes $x_1\to z_1$ and $x_2\to z_2$.  The subtensor $\zeta\circ G$ corresponds to mapping $f(\zeta^{-1}x;G) = f_1(z) + \frac{1}{4}f_3(z)$ as in Section \ref{what are inv n=2}.  Note from \eqref{f_1 eq} that $f_1(z) = \frac{3\|u_G\|}{4}z_1(z_1^2 + z_2^2)$ with respect to trilinear tensor $\E$ given by
$$
\E^1 = \frac{1}{4}\left(\begin{array}{cc} 3\|u_G\|&0\\0&\|u_G\|\end{array}\right), \ \ \E^2 =\frac{1}{4} \left(\begin{array}{cc} 0&\|u_G\| \\ \|u_G\|&0\end{array}\right).$$
The function $f_3(z)$ has trace $\frac{1}{6}\nabla \big(\Delta f_3\big)=0$ (recall that $\Delta$ is rotationally invariant), and corresponds to a trilinear tensor $\mathcal{D}$ where both $\mathcal{D}^1$, $\mathcal{D}^2$ have zero trace (cf. \eqref{D eq}).  Then, two parameters $4\gamma_1=\mathcal{D}^1_{2,2}$, $4\gamma_2=\mathcal{D}^2_{1, 1}$ determine $\mathcal{D}$.  Calling now $\alpha = \|u_G\|/4$, we have
  that the rotated subtensor $\zeta \circ G = \E + \frac{1}{4}\mathcal{D}$
 is in form
$$\left(\begin{array}{cc} 3 \alpha - \gamma_1& \gamma_2\\\gamma_2&\alpha+\gamma_1\end{array}\right), \ \ \left(\begin{array}{cc}\gamma_2&\alpha+\gamma_1\\\alpha+\gamma_1&-\gamma_2\end{array}\right).$$

If $\|u_G\|=0$, then ${\rm Trace}\; G^1 = {\rm Trace}\;  G^2 = 0$ so that $G$ is already in this form with $\alpha =0$.

Then, we say the `canonical' form for $\Gamma$, after putting it in reduced form $R$ and then rotating its $2\times 2\times 2$ subtensor to $\zeta\circ G$, is the following:
\begin{align}
\label{R tensor form}
\left(\begin{array}{ccc}3 \alpha-\gamma_1&\gamma_2&0\\
\gamma_2&\alpha+\gamma_1&0\\
0&0&0\end{array}\right), \ \left(\begin{array}{ccc}\gamma_2&\alpha+\gamma_1&0\\ \alpha+\gamma_1&-\gamma_2&0\\
0&0&0\end{array}\right), \ \left(\begin{array}{ccc}0&0&0\\0&0&0\\0&0&\beta_3\end{array}\right),
\end{align}
corresponding to
$$f(x;R) =( 3\alpha - \gamma_1)x_1^3 + 3\gamma_2x_1^2x_2 + 3(\alpha+\gamma_1)x_1x_2^2 -\gamma_2x_2^3 + \beta_3x_3^3.$$
We also calculate that 
\begin{align}
\label{R^2 form}
R^{*2} = \left(\begin{array}{ccc} 2(5\alpha^2-2\alpha\gamma_1 + \gamma_1^2+\gamma_2^2)& 4\alpha\gamma_2 & 0\\
4\alpha\gamma_2& 2\big((\alpha+\gamma_1)^2 + \gamma_2^2\big)& 0\\
0&0&\beta_3^2\end{array}\right).
\end{align}

The main reason we consider this type of canonical form for $\Gamma$ is that the later solvability condition \eqref{condition A} with respect to \eqref{PD inverting} allows lower degree and more succinct expression in Lemma \ref{lem:PD solvability} than if we employed the reduced form \eqref{partial reduced}.

It is clear that $\mathcal{J}=\{q_i\}$ composed of
\begin{align}
\label{PD inverting}
q_1 &= \beta_3^2\\
 q_2&=\|{\rm Trace} \; G\|^2=16\alpha^2\nonumber\\
 q_3 &= \det G^{*2} =  20\alpha^4 + 32\alpha^3\gamma_1+8\alpha^2\gamma_1^2+4\gamma_1^4 + 8\alpha^2\gamma_2^2+8\gamma_1^2\gamma_2^2+4\gamma_2^4\nonumber\\
 q_4&={\rm Trace}\; G^{*2}= 12\alpha^2+4\gamma_1^2+4\gamma_2^2\nonumber
 \end{align}
are $G_\RR$ invariants on $\mathcal{R}$ which form an integrity basis for the invariants of the $G_\RR$ action on $\mathcal{R}$ as they are independent (by a Groebner basis calculation) and fit the degree specifications. 
 
 Our task now is to extend these to invariants $\widetilde{\mathcal{J}}$ defined on an invariant set $\widetilde{\mathcal{T}}$ containing $\mathcal{S}$.  It not evident immediately how to extend the $\{q_i\}$.  However, let us try to rewrite them in turns of the $O(3)$ invariants $\mathcal{I}^+=\{H_2, H_4, J_2, L_4\}$ in the integrity basis $\mathcal{I}$ in \eqref{OA basis}.  We comment that this choice tries to match the degree structures in $\{q_i\}$, given there aren't $3$ invariants in $\mathcal{I}$ of degree $2$.  Also, we remark that not all choices of $4$ invariants in $\mathcal{I}$ would be amenable in the following calculations, because of syzygies among them.  However, the choice $\mathcal{I}^+$ will suffice here.
  
 When restricted to $\mathcal{R}$, these four invariants in $\mathcal{I}^+$ may be
 evaluated on canonical form tensors:
 \begin{align}
 \label{I^+ eq}
H_2&=10 q_1-15 q_2+25 q_4\\
H_4&=44 q_1^2-42 q_1 q_2+144 q_2^2-30 q_3+100 q_1 q_4-420 q_2 q_4+320 q_4^2\nonumber \\
J_2&=q_1+q_2\nonumber \\
L_4&=2 q_1^2-6
q_1 q_2+2 q_2^2-10 q_3-\frac{5}{2} q_2 q_4+\frac{5}{2} q_4^2. \nonumber
\end{align}
We denote $\mathcal{I}' =\big\{H_6, H_{10}, J_4, K_4, J_6, K_6, L_6, M_6, H_8\big\}$ as the remaining invariants in $\mathcal{I}$.

Let us now invert $\mathcal{I}^+$, and define 
 \begin{align}
 \label{eqn:partial solv}
\widetilde q_1 &= \frac{H_2^2 - 2H_4 - 3H_2J_2 + 6J_2^2 + 6L_4}{9(10 J_2 - H_2)}\\
\widetilde q_2&= J_2 - \widetilde q_1 \nonumber\\
\widetilde q_3&= \frac{1}{11250}\big(-8 H_2^2+25 H_4+60 H_2 J_2+1500 J_2^2-1200 L_4-11250 J_2 \widetilde q_1+11250 \widetilde q_1^2\big) \nonumber\\
\widetilde q_4&= \frac{1}{25}\big(H_2+ 15 J_2-25\widetilde q_1\big). \nonumber
\end{align}
We claim that $\widetilde{\mathcal{J}}=\{\widetilde q_i\}$ is a suitable extension of $\mathcal{J}$ of $O(3)$ invariants.  
Indeed, note that $10J_2-H_2 = 25(q_2-q_4)=100(\alpha^2-\gamma_1^2 -\gamma_2^2) \neq 0$
is the condition for $G^1, G^2$ not to be fully decoupleable (cf. Proposition \ref{prop:2.2} applied to the canonical form), and therefore $\Gamma$ to not be fully decoupleable (cf. Lemma \ref{rem:3.2}). 

Since the expressions for $\mathcal{J}=\{q_i\}$ are in terms of the integrity basis $\mathcal{I}$, they are $O(3)$ invariants where they are defined on $\mathcal{T}_3$.  Notice that the fraction in the equation for $q_1$ is well defined on the set $\mathcal{S}$ of partially but not fully decoupleable tensors in $\mathcal{T}_3$. Hence, $\widetilde{\mathcal{J}}$ is an extension of $\mathcal{J}$ to $O(3)$ invariants defined on the invariant set $\widetilde{\mathcal{T}} = \{\Gamma\in \mathcal{T}_3: 10J_2(\Gamma)\neq H_2(\Gamma)\}$ containing $\mathcal{S}$.

 To address solvability with respect to $\widetilde{\mathcal{J}}$, we prescribe the condition:
For a tensor $\Gamma\in \widetilde{\mathcal{T}}$, there is a canonical form tensor $R\in \mathcal{R}$ such that
\begin{align}
\label{condition A}
\widetilde q_i(\Gamma)=q_i(R),\ \ {\rm for \ }i=1,2,3,4.
\end{align}  

We rewrite this condition in terms of inequalities below in Lemma \ref{lem:PD solvability}.  We also show by Example \ref{exa:numerical} that the condition is needed, by specifying a tensor $\Gamma \in \widetilde{\mathcal{T}}$ which does not satisfy \eqref{condition A}.

\begin{thm}
\label{thm:partial n=3}
When $n=3$, there are $9$ explicit relations \eqref{PD relations} defined on $\widetilde{\mathcal{T}}=\{ \Ga  \in \mathcal{T}_3 :  H_2(\Ga) \neq 10 J_2(\Ga) \}$.
A tensor $\Gamma\in \mathcal{T}_3$ is partially but not fully decoupleable exactly when $\Gamma\in \widetilde{\mathcal{T}}$ and $\Gamma$ satisfies these relations, along with the solvability condition \eqref{condition A}, explicitly evaluated in Lemma \ref{lem:PD solvability}.
These relations on $\widetilde{\mathcal{T}}$ are as follows.
\begin{align}
\label{PD relations}
&18 J_4 =  -H_2^2+12 J_2 H_2-24 J_2^2+2 H_4+12 L_4, \\
  & 9 K_4 =  -2 H_2^2+15 J_2 H_2-66 J_2^2+4 H_4+6 L_4, \nonumber \\
   27 H_6& =  8100 \widetilde q_1^3-8100 J_2 \widetilde q_1^2-\left(7272 J_2^2-234 H_4+1512 L_4\right) \widetilde q_1-13 H_2^3+372 J_2^3\nonumber \\
   &-156 H_2 J_2^2  +26 H_2 H_4-7 H_2^2 J_2+146 H_4 J_2+30 H_2 L_4-924 J_2 L_4, \nonumber \\
   36 J_6& =  2700 \widetilde q_1^3-2700 J_2 \beta_3^4-\left(144 J_2^2-18 H_4+324 L_4\right) \widetilde q_1-H_2^3-72 J_2^3+12 H_2 J_2^2\nonumber \\
   &+2 H_2 H_4
   +6 H_4 J_2+12 H_2 L_4, \nonumber \\
   162 K_6& =  24300 \widetilde q_1^3-24300 J_2 \widetilde q_1^2-\left(8136 J_2^2-342 H_4+3456 L_4\right) \widetilde q_1-19 H_2^3-744 J_2^3\nonumber \\
   &+96 H_2 J_2^2 
     +38 H_2 H_4+14 H_2^2 J_2+86 H_4 J_2+48 H_2 L_4-744 J_2 L_4, \nonumber \\
   81 L_6& =  12150 \widetilde q_1^3-12150 J_2 \widetilde q_1^2-\left(7488 J_2^2-261 H_4+1998 L_4\right) \widetilde q_1-19 H_2^3-204 J_2^3\nonumber \\
   &-174 H_2 J_2^2 
     +38 H_2 H_4+23 H_2^2 J_2+149 H_4 J_2+48 H_2 L_4-852 J_2 L_4, \nonumber \\
   324 M_6& =  24300 \widetilde q_1^3-24300 J_2 \widetilde q_1^2+\left(5544 J_2^2-18 H_4-2376 L_4\right) \widetilde q_1+H_2^3-552 J_2^3\nonumber \\
   &+120 H_2 J_2^2 
     -2 H_2 H_4-14 H_2^2 J_2+22 H_4 J_2+6 H_2 L_4+420 J_2 L_4, \nonumber \\
   1458 H_8 & =  1895400 J_2 \widetilde q_1^3-\left(356400 J_2^2+40500 H_4-121500 L_4\right) \widetilde q_1^2 \nonumber \\
   &  +\left(158832 J_2^3+5796 H_4 J_2-206928 L_4 J_2\right) \widetilde q_1  +70 H_2^4+30792 J_2^4-8868 H_2 J_2^3\nonumber \\
   &+442 H_4^2 
     +830 H_2^2 J_2^2  -7216 H_4 J_2^2-11088 L_4^2-361 H_2^2 H_4+149 H_2^3 J_2\nonumber \\
     &+674 H_2 H_4 J_2 
     +30 H_2^2 L_4-3624 J_2^2 L_4+3828 H_4 L_4+408 H_2 J_2 L_4, \nonumber \\
   2187 H_{10}& =  \left(510300 H_4-5832000 J_2^2\right) \widetilde q_1^3-\left(4017600 J_2^3+251100 H_4 J_2+777600 L_4 J_2\right) \widetilde q_1^2 \nonumber \\
   &  -\big(9237600 J_2^4+95256 H_4 J_2^2+505440 L_4 J_2^2-10782 H_4^2+35640 L_4^2\nonumber \\
   &+71496 H_4 L_4\big) \widetilde q_1 
      +70 H_2^5-1226976 J_2^5+607632 H_2 J_2^4-159160 H_2^2 J_2^3\nonumber \\
      &+230888 H_4 J_2^3+1478 H_2 H_4^2 
    +110 H_2^3 J_2^2 +12212 H_2 H_4 J_2^2  +2448 H_2 L_4^2\nonumber \\
    &-35928 J_2 L_4^2-879 H_2^3 H_4-1161 H_2^4 J_2 
      +2506 H_4^2 J_2 -630 H_2^3 L_4  +2866 H_2^2 H_4 J_2 \nonumber \\
      &-423096 J_2^3 L_4 -125412 H_2 J_2^2 L_4+ 
     2040 H_2 H_4 L_4 +204 H_2^2 J_2 L_4-26160 H_4 J_2 L_4. \nonumber
 \end{align}
\end{thm}

  \begin{proof}We will apply Theorem \ref{thm:meta} and its scheme. 
  Since the $O(3)$ invariants $\mathcal{I}'$ on $\mathcal{T}_3$, when restricted to $\mathcal{R}$ are $G_\RR$ invariants on $\mathcal{R}$, we may express each of them in terms of the generating set $\mathcal{J}$, and so in terms of $\mathcal{I}^+$.  This procedure yields the
  $9$ relations listed above. 
  
  Hence, by Theorem \ref{thm:meta}, we conclude that a $\Gamma\in \mathcal{T}_3$ is partially but not fully decoupleable exactly when $\Gamma\in \widetilde{\mathcal{T}}$ and the $9$ relations \eqref{PD relations} hold, and
  the solvability condition \eqref{condition A} holds (which also identifies the remaining values of the integrity basis $\mathcal{I}^+$ on $\Gamma$ via \eqref{I^+ eq}). 
  \end{proof}

  We now address the solvability criterion \eqref{condition A}.
  \begin{lem}
  \label{lem:PD solvability}
  For $\Gamma\in \widetilde{\mathcal{T}}$, we have that \eqref{condition A} holds exactly when
  \begin{align*}
& \widetilde q_1(\Gamma)\geq 0, \ \ {\rm and \ \ }
  \widetilde q_2(\Gamma)\geq 0,
  \end{align*}
  and the following conditions depending on whether $\widetilde q_2(\Gamma) >0$ or $\widetilde q_2(\Gamma)=0$ hold:
  
  When  $\widetilde q_2(\Gamma)>0$, we must have
  \begin{align*}
&  -\widetilde q_2^4+4 \widetilde q_2^2 \widetilde q_3-16 \widetilde q_3^2+2 \widetilde q_2^3 \widetilde q_4-8 \widetilde q_2 \widetilde q_3 \widetilde q_4-2 \widetilde q_2^2 \widetilde q_4^2+8 \widetilde q_3 \widetilde q_4^2+2 \widetilde q_2 \widetilde q_4^3-\widetilde q_4^4 \geq 0.
  \end{align*}
  
  When $\widetilde q_2(\Gamma)=0$, we must have $\widetilde q_3, \widetilde q_4\geq 0$ and $\widetilde q_3 = (\widetilde q_4)^2/4$.
  \end{lem}
  
  \begin{proof}
  Consider the equations \eqref{PD inverting}, with $\{\tilde q_i\}$ in place of $\{q_i\}$.  The first equation imposes that $\widetilde q_1\geq 0$ in order to be able to take a square root to define $\beta_3$.
   The second equation then imposes that $\widetilde q_2 = 16 \alpha^2\geq 0$.
  
 \noindent {\it Case 1.}  Suppose that $\widetilde q_2>0$.   A Groebner basis calculation, eliminating $\alpha$ and $\gamma_1$, allows to obtain a linear equation for $\gamma_2^2$:
 $$ \widetilde q_2^4-4 \widetilde q_2^2 \widetilde q_3+16 \widetilde q_3^2-2 \widetilde q_2^3 \widetilde q_4+8 \widetilde q_2 \widetilde q_3 \widetilde q_4+2 \widetilde q_2^2 \widetilde q_4^2-8 \widetilde q_3 \widetilde q_4^2-2 \widetilde q_2 \widetilde q_4^3+\widetilde q_4^4+4 \widetilde q_2^3 \gamma _2^2=0.$$
 Hence, we obtain the fraction
 \begin{equation}
 \label{Solv num}
 \gamma_2^2 = \frac{1}{4\widetilde q_2^3}\big(-\widetilde q_2^4+4 \widetilde q_2^2 \widetilde q_3-16 \widetilde q_3^2+2 \widetilde q_2^3 \widetilde q_4-8 \widetilde q_2 \widetilde q_3 \widetilde q_4-2 \widetilde q_2^2 \widetilde q_4^2+8 \widetilde q_3 \widetilde q_4^2+2 \widetilde q_2 \widetilde q_4^3-\widetilde q_4^4\big),\end{equation}
  which must be nonnegative to be able to take a square root to define $\gamma_2$.  This means that  
  the numerator in \eqref{Solv num} must be nonnegative.  Moreover, we may solve for $\alpha$ now by taking the square root of $\widetilde q_2$.
  Finally, we may solve, without further conditions, for
  $$\gamma_1 = -\frac{\alpha  \left(\widetilde q_2^2-8 \widetilde q_3-2 \widetilde q_2 \widetilde q_4+2 \widetilde q_4^2\right)}{\widetilde q_2^2}.$$

  \noindent {\it Case 2.} Suppose $\widetilde q_2 = 0$.  From \eqref{PD inverting}, this means $\alpha = 0$.  Moreover, we must have $\widetilde q_4 = 4(\gamma_1^2 + \gamma^2_2)\geq 0$ and $\widetilde q_3 = 4(\gamma_1^2 + \gamma_2^2)^2\geq 0$ to be able to define $\gamma_1^2 + \gamma_2^2$.  In other words, we must have $\widetilde q_3=(\widetilde q_4)^2/4\geq 0$. 
  When $\widetilde q_3 = \widetilde q_4 = 0$, we have $\gamma_1 = \gamma_2=0$.  However, when $\widetilde q_4>0$ and so necessarily $\widetilde q_3>0$, although the sum $\gamma_1^2 + \gamma_2^2$ is determined, there is a continuum of choices $\gamma_1, \gamma_2$ which would satisfy the relation $\gamma_1^2 + \gamma_2^2 = (\widetilde q_4)/4$.

  Hence, taking the two cases into account, the conditions in the lemma are necessary, and also sufficient.
  \end{proof}
  
  We now give the example demonstrating that the solvability condition \eqref{condition A} is nontrivial on $\mathcal{T}_3$.
 \begin{exa}
\label{exa:numerical}
Let $\Gamma$ be the tensor corresponding to the cubic polynomial $f = 2 x_1^3+3 x_1^2 x_2+3 x_2^3-12 x_1 x_2 x_3+6 x_3^3$.

We can compute values of the integrity basis to get $H_2 = 1060, H_4 = 518384, J_2 = 56, L_4= -4528$. Since $H_2 \neq 10 J_2$, the quantities $\widetilde{q}_i$ for $i = 1,2, 3, 4$ are well defined. In particular $\widetilde{q}_1 = - 332/15$ so there is no real solution for the parameter $\beta_3$ which is given by $\beta_3^2 = \widetilde{q}_1$.
\end{exa}

\paragraph{{\bf Acknowledgements.}} We thank Jacopo De Nardis and Ali Zahra for pointing out difficulties with Proposition 8.1 in \cite{BFS} on characterization of when $\Gamma$ is partially decoupleable when $n=2$; in this regard, Lemma \ref{lem:2.1} and Proposition \ref{prop:2.2} correct its statement in \cite{BFS}.

We also thank Kirti Joshi and Klaus Lux for helpful discussions. 

\paragraph{{\bf Funding.}}
T.F. was supported by International Scientists Project of BJNSF \#IS23007.  S.S. was supported in part by a Simons Sabbatical grant.  S.V. was supported in part by NSF grant DMS-2108124.
 
 \paragraph{{\bf Conflicts of Interest.}} There are no competing financial or other interests with respect to this work.
 
 \paragraph{{\bf Data Availability Statement.}}
 The Mathematica code supporting calculations in this study are openly available in the public repository \href{https://github.com/shankar-cv/KPZ-Decouple}{https://github.com/shankar-cv/KPZ-Decouple}.

\end{document}